\def\DateTime{24/December/2011, 12:30 (Kyoto)}
\def\Version{Version $1.5$}
\def\yes{\if00}
\def\no{\if01}
\def\iftenpt{\no}
\def\ifelevenpt{\no}
\def\iftwelvept{\yes}

\def\ifusepdf{\no}
\def\ifpsfont{\no}
\def\iftxfont{\no}
\def\ifpxfont{\yes}
\def\ifeulerfont{\no}
\iftenpt
\documentclass[leqno]{amsart}
\else\fi
\ifelevenpt
\documentclass[leqno,11pt]{amsart}
\else\fi
\iftwelvept
\documentclass[leqno,12pt]{amsart}
\else\fi

\usepackage{amssymb}
\usepackage{amscd}
\usepackage{verbatim}
\usepackage{mathrsfs}

\ifusepdf
\usepackage{hyperref}
\else\fi
\ifpsfont
\usepackage[T1]{fontenc}
\usepackage{times}
\else\fi
\iftxfont
\usepackage{txfonts}
\else\fi
\ifpxfont
\usepackage{pxfonts}
\else\fi
\ifeulerfont
\usepackage{euler}
\else\fi

\input xy
\xyoption{all}

\ifelevenpt
\else\fi

\iftwelvept
\else\fi

\theoremstyle{plain}
\newtheorem{Theorem}{Theorem}[section]
\newtheorem{Proposition}[Theorem]{Proposition}
\newtheorem{Lemma}[Theorem]{Lemma}

\newtheorem{Corollary}[Theorem]{Corollary}

\newtheorem{Claim}{Claim}[Theorem]

\theoremstyle{definition}

\newtheorem{Remark}[Theorem]{Remark}

\def\rom{\textup}
\newcommand{\ZZ}{{\mathbb{Z}}}
\newcommand{\QQ}{{\mathbb{Q}}}
\newcommand{\RR}{{\mathbb{R}}}
\newcommand{\KK}{{\mathbb{K}}}
\newcommand{\CC}{{\mathbb{C}}}
\newcommand{\PP}{{\mathbb{P}}}

\newcommand{\FF}{{\mathbb{F}}}
\newcommand{\OO}{{\mathscr{O}}}

\newcommand{\Proj}{\operatorname{Proj}}

\newcommand{\Rat}{\operatorname{Rat}}

\newcommand{\lex}{\operatorname{lex}}

\newcommand{\Spec}{\operatorname{Spec}}

\newcommand{\Supp}{\operatorname{Supp}}

\newcommand{\mult}{\operatorname{mult}}

\newcommand{\ord}{\operatorname{ord}}

\newcommand{\adeg}{\widehat{\operatorname{deg}}}

\newcommand{\Div}{\operatorname{Div}}
\newcommand{\aDiv}{\widehat{\operatorname{Div}}}

\newcommand{\vol}{\operatorname{vol}}
\newcommand{\avol}{\widehat{\operatorname{vol}}}
\newcommand{\aH}{\hat{H}^0}

\newcommand{\Tpsh}{\operatorname{PSH}}

\newcommand{\Tqpsh}{\operatorname{QPSH}}

\newcommand{\rest}[2]{\left.{#1}\right\vert_{{#2}}}
\begin{document}

%%%%%%%%%%%
%% Title %%
%%%%%%%%%%%
\title[Arithmetic linear series with base conditions]%
{Arithmetic linear series with base conditions}
\author{Atsushi Moriwaki}
\address{Department of Mathematics, Faculty of Science,
Kyoto University, Kyoto, 606-8502, Japan}
\email{moriwaki@math.kyoto-u.ac.jp}
\date{\DateTime, (\Version)}
%\keywords{}
\begin{abstract}
In this note, we study the volume of arithmetic linear series with base conditions.
As an application, we consider the problem of Zariski decompositions on arithmetic varieties.
\end{abstract}

%%%%%
% 1st Draft
%%%%%
%\vskip -1cm
%\hfill\fbox{{\large\bf Draft for \Version}} %
%\footnote{I welcome any comments and suggestions. I will change several materials in the official version.}\par
%\vskip .5cm
%%%%%

\maketitle

%\setcounter{tocdepth}{2}
%\tableofcontents
%\newpage

\section*{Introduction}
Let $X$ be a projective and flat integral scheme over $\ZZ$.
We assume that $X$ is normal and the generic fiber of $X \to \Spec(\ZZ)$ is a
$d$-dimensional smooth variety over $\QQ$.
%Let $X$ be a $(d+1)$-dimensional, generically smooth, normal, projective arithmetic variety,
%that is,
%(i) $X$ is a projective and flat scheme over $\ZZ$, (ii)
%$X$ is integral and normal,  and  (iii) the generic fiber of $X \to \Spec(\ZZ)$ is a
%$d$-dimensional smooth variety over $\QQ$.
Let $\Div(X)$ be the group of Cartier divisors on $X$ and let 
$\Div(X)_{\RR} := \Div(X) \otimes_{\ZZ} \RR$.
%,whose element is called an {\em $\RR$-Cartier divisor on $X$}.
A pair $\overline{D} = (D, g)$ 
%of an $\RR$-Cartier divisor $D = \sum_{i=1}^r a_i D_i$
%and a locally integrable function $g : X(\CC) \to \RR \cup \{\pm\infty\}$ 
is called an arithmetic $\RR$-Cartier divisor
of $C^{0}$-type if
$D  \in \Div(X)_{\RR}$ (i.e. $D= \sum_{i=1}^r a_i D_i$ for some $D_1, \ldots, D_r \in \Div(X)$
and
$a_1, \ldots, a_r \in \RR$), % is an $\RR$-Cartier divisor on $X$,
$g : X(\CC) \to \RR \cup \{\pm\infty\}$ is a locally integrable function %$g$ is 
invariant under the complex conjugation map $F_{\infty} : X(\CC) \to X(\CC)$ and,
for any point $x \in X(\CC)$, there are an open neighborhood $U_{x}$ of $x$ and
a continuous function $u_x$ over $U_x$ such that
\[
g = u_x + \sum_{i=1}^r (-a_i) \log \vert f_i \vert^2\quad(a.e.)
\]
on $U_x$,
where $f_i$ is a local equation of $D_i$ on $U_x$ for each $i$.
We denote the vector space consisting of arithmetic $\RR$-Cartier divisors 
of $C^{0}$-type by $\aDiv_{C^{0}}(X)_{\RR}$.

Let $\Rat(X)^{\times}_{\RR} := \Rat(X)^{\times} \otimes_{\ZZ} \RR$ and let
\[
\widehat{(\ )}_{\RR}: \Rat(X)^{\times}_{\RR} \to \aDiv_{C^{0}}(X)_{\RR}
\]
be the natural extension of the homomorphism $\Rat(X)^{\times} \to \aDiv_{C^{0}}(X)_{\RR}$ given by
$\phi \mapsto \widehat{(\phi)}$.
Let $\overline{D}$ be an arithmetic $\RR$-Cartier divisor of $C^0$-type on $X$.
We define 
%$\widehat{\Gamma}^{\times}(X, \overline{D})$ and $\widehat{\Gamma}^{\times}_{\RR}(X, \overline{D})$
$\aH(X, \overline{D})$ and $\aH_{\RR}(X, \overline{D})$
to be
%\[
%\begin{cases}
%\widehat{\Gamma}^{\times}(X, \overline{D}) := \left\{ \phi \in \Rat(X)^{\times} \mid
%\overline{D} + \widehat{(\phi)} \geq (0,0) \right\}, \\
%\widehat{\Gamma}^{\times}_{\RR}(X, \overline{D}) := \left\{ \phi \in \Rat(X)^{\times}_{\RR} \mid
%\overline{D} + \widehat{(\phi)}_{\RR} \geq (0,0) \right\}.
%\end{cases}
%\]
\[
\begin{cases}
\aH(X, \overline{D}) := \left\{ \phi \in \Rat(X)^{\times} \mid
\overline{D} + \widehat{(\phi)} \geq (0,0) \right\} \cup \{ 0 \}, \\
\aH_{\RR}(X, \overline{D}) := \left\{ \phi \in \Rat(X)^{\times}_{\RR} \mid
\overline{D} + \widehat{(\phi)}_{\RR} \geq (0,0) \right\} \cup \{ 0 \}.
\end{cases}
\]
For $\xi \in X$, the {\em $\RR$-asymptotic multiplicity 
of $\overline{D}$ at $\xi$} is given by
%\[
%\mu_{\RR,\xi}(\overline{D}) := \begin{cases}
%\inf \left\{ \mult_{\xi}(D + (\phi)_{\RR}) \mid
%\phi \in \widehat{\Gamma}_{\RR}^{\times}(X, \overline{D}) \right\} & \text{if $\widehat{\Gamma}^{\times}_{\RR}(X, \overline{D}) \not= \emptyset$}, \\
%\infty & \text{otherwise},
%\end{cases}
%\]
\[
\mu_{\RR,\xi}(\overline{D}) := \begin{cases}
\inf \left\{ \mult_{\xi}(D + (\phi)_{\RR}) \mid
\phi \in \aH_{\RR}(X, \overline{D}) \setminus \{ 0 \} \right\} & \text{if $\aH_{\RR}(X, \overline{D}) \not= \{ 0 \}$}, \\
\infty & \text{otherwise},
\end{cases}
\]
where $\mult_{\xi}$ is the multiplicity of the local ring given by
a local equation (for details, see \cite[Section~2.8]{Ko} or \cite[SubSection~6.5]{MoArZariski}).
Moreover, for $\xi_1, \ldots, \xi_l \in X$ and $\mu_1, \ldots, \mu_l \in \RR_{\geq 0}$,
we define the {\em arithmetic linear series of $\overline{D}$ 
with base conditions $\mu_1 \xi_1, \ldots, \mu_l \xi_l$} to be
%\[
%\aH(X, \overline{D} ; \mu_1 \xi_1, \ldots, \mu_l \xi_l) := 
%\left\{ \phi \in \widehat{\Gamma}^{\times}(X, \overline{D}) \mid
%\mult_{\xi_i}(D + (\phi)) \geq \mu_i\ (\forall i) \right\} \cup \{ 0 \}.
%\]
\[
\aH(X, \overline{D} ; \mu_1 \xi_1, \ldots, \mu_l \xi_l) := 
\left\{ \phi \in \aH(X, \overline{D}) \setminus \{ 0 \} \mid
\mult_{\xi_i}(D + (\phi)) \geq \mu_i\ (\forall i) \right\} \cup \{ 0 \}.
\]
In addition, its volume is given by
\[
\avol(\overline{D} ; \mu_1 \xi_1, \ldots, \mu_l \xi_l) := \limsup_{n\to\infty}
\frac{\log \# \aH(X, n\overline{D} ; n\mu_1 \xi_1, \ldots, n\mu_l \xi_l)}{n^{d+1}/(d+1)!}.
\]

The main result of this paper is the following theorem:

\begin{Theorem}
\label{thm:vol:base:cond}
If $\xi_1, \ldots, \xi_l \in X_{\QQ}$, $\overline{D}$ is big and 
$\mu_i > \mu_{\RR,\xi_i}(\overline{D})$ for some $i$,
then 
\[
\avol(\overline{D} ; \mu_1 \xi_1, \ldots, \mu_l \xi_l) < \avol(\overline{D}).
\]
\end{Theorem}

Let us introduce a Zariski decomposition on $X$.
A decomposition $\overline{D} = \overline{P} + \overline{N}$ is called
a Zariski decomposition of $\overline{D}$ if the following conditions are satisfied
(for the positivity of arithmetic $\RR$-Cartier divisors of $C^0$-type, see 
Conventions and terminology~\ref{CT:positivity:arithmetic:divisors}):
\begin{enumerate}
\renewcommand{\labelenumi}{(\arabic{enumi})}
\item
$\overline{P}$ is a nef arithmetic $\RR$-Cartier divisor of $C^0$-type.

\item
$\overline{N}$ is an effective arithmetic $\RR$-Cartier divisor of $C^0$-type.

\item
$\avol(\overline{P}) = \avol(\overline{D})$.
\end{enumerate}
It is easy to see that the existence of a Zariski decomposition of $\overline{D}$ implies 
the pseudo-effectivity of $\overline{D}$ (cf. SubSection~\ref{subsec:imp:Zariski}).
Let $\Upsilon(\overline{D})$ be the set of all nef arithmetic $\RR$-Cartier divisors $\overline{M}$
of $C^0$-type
with $\overline{M} \leq \overline{D}$. 
%The main theorems of the paper \cite[Theorem~9.2.1 and Theorem~9.3.4]{MoArZariski} say that
%if $d=1$, $X$ is regular and $\Upsilon(\overline{D}) \not= \emptyset$,
%then there is the greatest element $\overline{P}$ of $\Upsilon(\overline{D})$ and
%$\avol(\overline{P}) = \avol(\overline{D})$.
Note that 
$\Upsilon(\overline{D})$ is not empty
if and only if there is a decomposition $\overline{D} = \overline{P} + \overline{N}$ with
the conditions (1) and (2).
For a non-big pseudo-effective arithmetic $\RR$-Cartier divisor $\overline{D}$ of $C^0$-type, 
the non-emptyness of $\Upsilon(\overline{D})$ is
a non-trivial problem. It is closely related to the fundamental question raised in the paper
\cite{MoD}.
Further, in the case where $d=1$ and $X$ is regular,
once we can see $\Upsilon(\overline{D}) \not= \emptyset$,
%then, 
%\cite[Theorem~9.2.1 and Theorem~9.3.4]{MoArZariski},
%\cite[Theorem~9.2.1]{MoArZariski},
the greatest element of $\Upsilon(\overline{D})$ is ensured
by the main theorem of the paper \cite[Theorem~9.2.1]{MoArZariski}, and
it turns out to be the actual positive part of $\overline{D}$
(cf. Remark~\ref{rem:def:zariski:decomp}).
In this sense, the above definition has a meaning even for
a non-big pseudo-effective arithmetic $\RR$-Cartier divisor.
Of course, in this case, the uniqueness of the decomposition is not
guaranteed.

We would like to apply the above theorem to 
the problem of Zariski decompositions on arithmetic varieties
(cf. Conventions and terminology~\rom{\ref{CT:arith:var}}).
The first one is an estimation of the asymptotic multiplicity.

\begin{Theorem}
We assume that $\overline{D}$ is big.
If $\overline{D} = \overline{P} + \overline{N}$ is 
a Zariski decomposition of $\overline{D}$, 
then $\mu_{\RR,\xi}(\overline{D}) = \mult_{\xi}(N)$ for all $\xi \in X_{\QQ}$.
\end{Theorem}

In the paper \cite{MoBig}, we considered a decomposition
$\overline{D} = \overline{P} + \overline{N}$ such that
$\mu_{\RR,\Gamma}(\overline{D}) = \mult_{\Gamma}(N)$ for any horizontal prime divisor 
$\Gamma$ on $X$.
The above theorem means that a Zariski decomposition of a big
arithmetic $\RR$-Cartier divisor of $C^0$-type yields the decomposition
treated in \cite[Section~5]{MoBig} (cf. Remark~\ref{rem:def:zariski:decomp}).
Thus, as a corollary, we have the following variant of the impossibility of Zariski decompositions.
%, which is much clearer than \cite[Theorem~5.6]{MoBig}.
The condition  $\mu_{\RR,\Gamma}(\overline{D}) = \mult_{\Gamma}(N)$ 
is rather technical, 
so that this form seems to be  more acceptable than \cite[Theorem~5.6]{MoBig}.

\begin{Theorem}
We suppose that $d \geq 2$ and $X = \PP^d_{\ZZ} (:= \Proj(\ZZ[T_0, T_1, \ldots, T_d]))$.
In addition, we assume that $\overline{D}$ is given by
\[
\left(H_0, \log (a_0 + a_1 \vert z_1 \vert^2 + \cdots + a_d \vert z_d \vert^2) \right),
\]
where $H_0 := \{ T_0 = 0 \}$,
$z_i := T_i/T_0$ \rom{(}$i=1, \ldots, d$\rom{)} and $a_0, a_1, \ldots, a_d \in \RR_{>0}$.
If $\overline{D}$ is big and not nef 
\rom{(}i.e. $a_0 + \cdots + a_d > 1$ and
$a_i < 1$ for some $i$\rom{)},
then, for any birational morphism $f : Y \to \PP^d_{\ZZ}$ of generically smooth,
normal and projective arithmetic varieties
\rom{(}cf. Conventions and terminology~\rom{\ref{CT:arith:var}}\rom{)},
there is no Zariski decomposition of $f^*(\overline{D})$ on $Y$.
\end{Theorem}

\medskip
%In the case where $d = 1$ and $X$ is regular, we can give an alternative definition of
%Zariski decompositions.
%A decomposition $\overline{D} = \overline{P} + \overline{N}$ is called
%a Zariski decomposition of $\overline{D}$ in the sense of \cite{MoArZariski}
%if $\overline{P}$ gives rise to the greatest element of $\Upsilon(\overline{D})$
%(cf. Remark~\ref{rem:def:zariski:decomp}).
%
The third application is the unique existence of Zariski decompositions of big
arithmetic $\RR$-Cartier divisors on arithmetic surfaces.
The new point is the uniqueness of the Zariski decomposition in the sense of this paper, which gives a characterization of the Zariski decompositions.

\begin{Theorem}
We assume that $d=1$ and $X$ is regular.
If $\overline{D}$ is big, then there exists a unique 
Zariski decomposition of $\overline{D}$.
Namely, 
%the Zariski decomposition in the sense of \cite{MoArZariski}
%coincides with that of $\overline{D}$ in this paper.
the positive part of the Zariski decomposition of $\overline{D}$
is the greatest element of $\Upsilon(\overline{D})$.
\end{Theorem}

Finally I would like to give my hearty thanks to the referee for pointing out 
inadequate parts of this paper.

\renewcommand{\thesubsubsection}{\arabic{subsubsection}}

\bigskip
\renewcommand{\theequation}{CT.\arabic{subsubsection}.\arabic{Claim}}
\subsection*{Conventions and terminology}
%We basically use the same notation as in \cite{MoArZariski}.
Here we fix several conventions and the terminology of this paper.
Let $\KK$ be either $\QQ$ or $\RR$.
For details of \ref{CT:arith:div} and \ref{CT:positivity:arithmetic:divisors},
see \cite{MoArZariski}.

%Moreover, in the following 2 $\sim$ 5,
%$X$ is a $d$-dimensional, generically smooth, normal and projective arithmetic variety.

\subsubsection{}
\label{CT:arith:var}
An {\em arithmetic variety} means a quasi-projective and flat integral scheme over $\ZZ$.
An arithmetic variety is said to be
{\em generically smooth} if the generic fiber over $\ZZ$ is smooth over
$\QQ$.

\subsubsection{}
\label{CT:arith:div}
Let $X$ be a generically smooth and normal arithmetic variety.
Let $\Div(X)$ be the group of Cartier divisors on $X$ and let
$\Div(X)_{\KK} := \Div(X) \otimes_{\ZZ} \KK$, whose element is
called a {\em $\KK$-Cartier divisor on $X$}.
A pair $\overline{D} = (D, g)$ 
%of an element $D = \sum_{i=1}^r a_i D_i$ of $\Div(X)_{\KK}$
%and a locally integrable function $g : X(\CC) \to \RR \cup \{\pm\infty\}$ 
is called an {\em arithmetic $\KK$-Cartier divisor
of $C^{\infty}$-type} (resp. {\em of $C^{0}$-type}) if
the following conditions are satisfied:
\begin{enumerate}
\renewcommand{\labelenumi}{(\alph{enumi})}
\item
$D$ is a $\KK$-Cartier divisor on $X$, that is,
$D = \sum_{i=1}^r a_i D_i$  for some $D_1, \ldots, D_r \in \Div(X)$ and
$a_1, \ldots, a_r \in \KK$.

\item
$g : X(\CC) \to \RR \cup \{\pm\infty\}$ is a locally integrable function and
$g \circ F_{\infty} = g \ (a.e.)$, where $F_{\infty} : X(\CC) \to X(\CC)$
is the complex conjugation map.

\item
For any point $x \in X(\CC)$, there are an open neighborhood $U_{x}$ of $x$ and
a $C^{\infty}$-function (resp. continuous function) 
$u_x$ on $U_x$ such that
\[
g = u_x + \sum_{i=1}^r (-a_i) \log \vert f_i \vert^2\quad(a.e.)
\]
on $U_x$,
where $f_i$ is a local equation of $D_i$ over $U_x$ for each $i$.
\end{enumerate}
If $u_x$ can be taken as a continuous plurisubharmonic function over $U_x$ for all $x \in X(\CC)$,
then the pair $\overline{D}$ is called an {\em arithmetic $\KK$-Cartier divisor
of $(C^0 \cap \Tpsh)$-type}.
Let $\mathcal{C}$ be either $C^{\infty}$ or $C^0$ or $C^0 \cap \Tpsh$.
%Moreover, if $D \in \Div(X)$,
%then $\overline{D}$ is called an {\em arithmetic Cartier divisor}.
The set of all arithmetic $\KK$-Cartier divisors of $\mathcal{C}$-type is denoted by
$\aDiv_{\mathcal{C}}(X)_{\KK}$. Moreover, the set 
\[
\left\{ (D, g) \in \aDiv_{\mathcal{C}}(X)_{\KK}
\mid D \in \Div(X) \right\}
\]
is denoted by
$\aDiv_{\mathcal{C}}(X)$. An element of $\aDiv_{\mathcal{C}}(X)$ is called
an {\em arithmetic Cartier divisor
of $\mathcal{C}$-type}.
Note that $\aDiv_{C^{\infty}}(X)_{\KK}$ and $\aDiv_{C^{0}}(X)_{\KK}$ are
vector spaces over $\KK$ and $\aDiv_{C^{0} \cap \Tpsh}(X)_{\KK}$ forms a cone in
$\aDiv_{C^{0}}(X)_{\KK}$.
For $\overline{D} = (D, g), \overline{E} = (E, h) \in \aDiv_{C^0}(X)_{\KK}$,
we define relations
$\overline{D} = \overline{E}$ and $\overline{D} \geq \overline{E}$ as follows:
\begin{align*}
\overline{D} = \overline{E} & \quad\overset{\text{def}}{\Longleftrightarrow}\quad D = E, \ \  g = h \ (a.e.), \\
\overline{D} \geq \overline{E} & \quad\overset{\text{def}}{\Longleftrightarrow}\quad D \geq E, \ \  g \geq h \ (a.e.).
\end{align*}
\begin{comment}
Let $\Rat(X)^{\times}_{\KK} := \Rat(X)^{\times} \otimes_{\ZZ} \KK$, and
let 
\[
(\ )_{\KK} : \Rat(X)^{\times}_{\KK} \to \Div(X)_{\KK}\quad\text{and}\quad
\widehat{(\ )}_{\KK} : \Rat(X)^{\times}_{\KK} \to \aDiv_{C^{\infty}}(X)_{\KK}
\]
 be the natural extensions of
the homomorphisms 
\[
\Rat(X)^{\times} \to \Div(X)\quad\text{and}\quad
\Rat(X)^{\times} \to \aDiv_{C^{\infty}}(X)
\]
given by
$\phi \mapsto (\phi)$ and 
$\phi \mapsto \widehat{(\phi)}$ respectively.
Let $\overline{D}$ be an arithmetic $\RR$-Cartier divisor of $C^0$-type.
%(for details, see \cite{MoArZariski}).
We define $\widehat{\Gamma}^{\times}(X, \overline{D})$ and $\widehat{\Gamma}^{\times}_{\KK}(X, \overline{D})$
to be
\[
\begin{cases}
\widehat{\Gamma}^{\times}(X, \overline{D}) := \left\{ \phi \in \Rat(X)^{\times} \mid
\overline{D} + \widehat{(\phi)} \geq (0,0) \right\}, \\
\widehat{\Gamma}^{\times}_{\KK}(X, \overline{D}) := \left\{ \phi \in \Rat(X)^{\times}_{\KK} \mid
\overline{D} + \widehat{(\phi)}_{\KK} \geq (0,0) \right\}.
\end{cases}
\]
Note that $\widehat{\Gamma}^{\times}_{\QQ}(X, \overline{D}) = \bigcup_{n=1}^{\infty} \widehat{\Gamma}^{\times}(X, n\overline{D})^{1/n}$.
\end{comment}

\subsubsection{}
\label{CT:positivity:arithmetic:divisors}
Let $X$ be a generically smooth, normal  and projective arithmetic variety.
Let $\overline{D}$ be an arithmetic $\RR$-Cartier divisor of $C^0$-type on $X$.
The effectivity, bigness, pseudo-effectivity and nefness of $\overline{D}$ are defined as follows:
\begin{enumerate}
\item[$\bullet$]
$\overline{D}$ is effective
$\quad\overset{\text{def}}{\Longleftrightarrow}\quad$
$\overline{D} \geq (0, 0)$.

\item[$\bullet$] $\overline{D}$ is big $\quad\overset{\text{def}}{\Longleftrightarrow}\quad$ $\avol(\overline{D}) > 0$.

\item[$\bullet$] $\overline{D}$ is pseudo-effective $\quad\overset{\text{def}}{\Longleftrightarrow}\quad$ $\overline{D}+\overline{A}$
is big for any big arithmetic $\RR$-Cartier divisor $\overline{A}$ of $C^0$-type.

\item[$\bullet$] $\overline{D} = (D,g)$ is nef $\quad\overset{\text{def}}{\Longleftrightarrow}$ 
\par
\begin{enumerate}
\renewcommand{\labelenumii}{(\alph{enumii})}
\item
$\adeg(\rest{\overline{D}}{C}) \geq 0$ for all reduced and irreducible $1$-dimensional closed subschemes $C$ of $X$.

\item
$\overline{D}$ is of $(C^0 \cap \Tpsh)$-type.
\end{enumerate}
\end{enumerate}
%Note that a nef arithmetic $\RR$-Cartier divisor of $C^0$-type is
%pseudo-effective (cf. \cite[Proposition~6.2.1, Proposition~6.2.2 and Proposition~6.3.2]{MoArZariski}).
The interrelations of the various types of positivity as above can be summarized as follows:
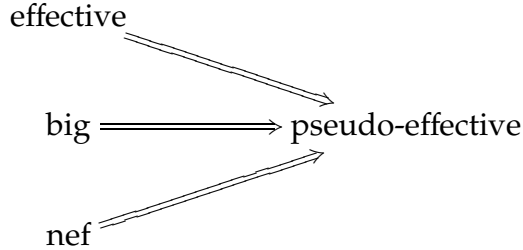
\[
\xymatrix{
\text{effective} \ar@{=>}[rrd] & & \\
\text{big} \ar@{=>}[rr] & & \text{pseudo-effective} \\
\text{nef} \ar@{=>}[rru] & & 
}
\]

\renewcommand{\theTheorem}{\arabic{section}.\arabic{subsection}.\arabic{Theorem}}
\renewcommand{\theClaim}{\arabic{section}.\arabic{subsection}.\arabic{Theorem}.\arabic{Claim}}
\renewcommand{\theequation}{\arabic{section}.\arabic{subsection}.\arabic{Theorem}.\arabic{Claim}}

\section{Generalizations of Boucksom-Chen's results to $\RR$-Cartier divisors}

In this section, we will give generalizations of Boucksom-Chen's results \cite{BC} to 
arithmetic $\RR$-Cartier divisors. 
All results in this section can be proved in the similar way as the paper \cite{BC}.

\subsection{Geometric case}
First of all, let us review the geometric case.
The contents of this subsection are generalizations of the works due to
Okounkov \cite{O}, Lazarsfeld-Musta\c{t}\u{a} \cite{LM} and
Kaveh-Khovanskii \cite{KK1}, \cite{KK2} to $\RR$-Cartier divisors.

Let $T$ be a $d$-dimensional, geometrically irreducible, normal and
projective variety over a field $F$.
Let $\overline{F}$ be an algebraic closure of $F$ and let
$T_{\overline{F}} := T \times_{\Spec(F)} \Spec(\overline{F})$.
Let $P \in T(\overline{F})$ be a regular point and let $z_P = (z_1, \ldots, z_d)$ be a local system of parameters of $\OO_{T_{\overline{F}}, P}$.
Then
\[
\widehat{\OO}_{T_{\overline{F}}, P} = \overline{F}[\![z_1, \ldots, z_d]\!],
\]
where $\widehat{\OO}_{T_{\overline{F}}, P}$ is the completion of 
$\OO_{T_{\overline{F}}, P}$ with respect to
the maximal ideal of $\OO_{T_{\overline{F}}, P}$.
Thus, for $f \in \OO_{T_{\overline{F}}, P}$, we can put
\[
f = \sum_{(a_1, \ldots, a_d) \in \ZZ_{\geq 0}^d} c_{(a_1, \ldots, a_d)} z_1^{a_1} \cdots z_d^{a_d},
\]
where $c_{(a_1, \ldots, a_d)} \in \overline{F}$.
Note that $\ZZ^d$ has the lexicographic order $<_{\lex}$, that is,
\[
(a_1,\ldots, a_d) <_{\lex} (b_1, \cdots, b_d) \quad\overset{\text{def}}{\Longleftrightarrow}\quad
\text{$a_1 = b_1, \ldots, a_{i-1} = b_{i-1}, a_{i} < b_{i}$ for some $i$}.
\]
We define $\ord_{z_P}(f)$ to be
\[
\ord_{z_P}(f) := \begin{cases}
\min_{<_{\lex}} \left\{ (a_1, \ldots, a_d) \mid c_{(a_1, \ldots, a_d)} \not= 0 \right\} & \text{if $f \not= 0$},\\
\infty & \text{otherwise},
\end{cases}
\]
%Note that $\ord_{z_P}$ satisfies the following properties:
which gives rise to a rank $d$ valuation, that is,
the following properties are satisfied:
\begin{enumerate}
\renewcommand{\labelenumi}{(\roman{enumi})}
\item
$\ord_{z_P}(fg) = \ord_{z_P}(f) + \ord_{z_P}(g)$ for $f, g \in \OO_{T_{\overline{K}},P}$.

\item
$\ord_{z_P}(f + g) \geq \min \{ \ord_{z_P}(f), \ord_{z_P}(g) \}$ for $f, g \in \OO_{T_{\overline{F}},P}$.
\end{enumerate}
By the property (i), $\ord_{z_P} : \OO_{T_{\overline{F}},P} \setminus \{ 0 \} \to \ZZ^d$ has the natural extension
\[
\ord_{z_P} : \Rat(T_{\overline{F}})^{\times} \to \ZZ^d
\]
given by $\ord_{z_P}(f/g) = \ord_{z_P}(f) - \ord_{z_P}(g)$.
As $\ord_{z_P}(u) = (0,\ldots,0)$ for all $u \in \OO^{\times}_{T_{\overline{F}},P}$, 
$\ord_{z_P}$ induces
$\Rat(T_{\overline{F}})^{\times}/\OO^{\times}_{T_{\overline{F}},P} \to \ZZ^d$. 
The composition of homomorphisms
\[
\Div(T_{\overline{F}}) \overset{\alpha_P}{\longrightarrow} \Rat^{\times}(T_{\overline{F}})/\OO^{\times}_{T_{\overline{F}},P} \overset{\ord_{z_P}}{\longrightarrow} \ZZ^d
\]
is
denoted by $\mult_{z_P}$, where 
$\Div(T_{\overline{F}})$ is the group of Cartier divisors on $T_{\overline{F}}$ and
$\alpha_P : \Div(T_{\overline{F}}) \to \Rat(T_{\overline{F}})^{\times}/\OO^{\times}_{T_{\overline{F}},P}$ is the natural
homomorphism. Moreover, the homomorphism $\mult_{z_P} : \Div(T_{\overline{F}}) \to \ZZ^d$
yields the natural extension
\[
\Div(T_{\overline{F}}) \otimes_{\ZZ} \RR \to \RR^d
\]
over $\RR$.
By abuse of notation, the above extension is also denoted by $\mult_{z_P}$.

\medskip
For $D \in \Div(T)_{\RR} := \Div(T) \otimes_{\ZZ} \RR$, 
let $H^0(T, D)$ be a vector space over $F$ given by
\[
H^0(T, D) := \{ \phi \in \Rat(T)^{\times} \mid (\phi) + D \geq 0 \} \cup \{ 0 \}.
\]
In the same way as  \cite[Lemma~1.3]{LM} or \cite[(1.1)]{BC}],
we can see
\[
\dim_F V = \# \left\{ \mult_{z_{P}}((\phi) + D_{\overline{F}}) \mid
\phi \in V \otimes_{F} \overline{F} \setminus \{ 0 \} \right\}
\]
for a subspace $V$ of $H^0(T, D)$.

We set $R(D) := \bigoplus_{m \geq 0} H^0(T, mD)$, which forms a graded algebra in the natural way.
Let $V_{\bullet}$ be a graded subalgebra of $R(D)$.
We say {\em $V_{\bullet}$ contains an ample series} if 
$V_m \not= \{ 0\}$ for $m \gg 1$ and 
there is an ample $\QQ$-Cartier divisor $A$ with 
the following properties:
\[
\begin{cases}
\bullet\ \text{$A \leq D$.} \\
\bullet\ \text{There is a positive integer $m_0$ such that
$H^0(T, mm_0A) \subseteq V_{mm_0}$ for all $m \geq 1$.}
\end{cases}
\]
We set 
\[
\Gamma(V_{\bullet}) = \bigcup_{V_m \not= \{ 0 \}, m \geq 0} \left\{ (\mult_{z_P}((\phi) + mD_{\overline{F}}), m) \in \RR_{\geq 0}^{d} \times \ZZ_{\geq 0} \mid \phi \in V_{m} \otimes_F \overline{F} \setminus \{ 0 \} \right\}.
\]
Let $v : \RR^{d+1} \to \RR^d$ and $h : \RR^{d+1} \to \RR$ be the projections given by
\[
v(x_1, \ldots, x_d, x_{d+1}) = (x_1, \ldots, x_d)\quad\text{and}\quad
h(x_1, \ldots, x_d, x_{d+1}) = x_{d+1}.
\]
Let $\Theta$ be an effective $\RR$-Cartier divisor such that $D + \Theta \in \Div(T)$.
We assume that $V_{\bullet}$ contains an ample linear series.
Then, in the same way as \cite[Lemma~2.2]{LM},
we can see the following:
% we can easily to see the following \cite[Lemma~2.2]{LM}:
\begin{enumerate}
\renewcommand{\labelenumi}{(\arabic{enumi})}
\item
If we set $\theta = \mult_{z_P}(\Theta_{\overline{F}})$ and
$\Gamma'(V_{\bullet}) = \{ \gamma + h(\gamma)(\theta, 0) \mid \gamma \in \Gamma(V_{\bullet}) \}$,
then $\Gamma'(V_{\bullet}) \subseteq \ZZ_{\geq 0}^{d+1}$ and
$\Gamma'(V_{\bullet})$ generates $\ZZ^{d+1}$ as a $\ZZ$-module.

\item
${\displaystyle \bigcup_{m > 0} \frac{1}{m} \Gamma(V_{\bullet})_m}$ is bounded in $\RR^d$, where
\[
\Gamma(V_{\bullet})_m : = v\left(\Gamma(V_{\bullet}) \cap ( \RR_{\geq 0}^{d} \times \{ m \})\right) = v \left(\{ \gamma \in \Gamma(V_{\bullet}) \mid h(\gamma) = m \}\right).
\]
\end{enumerate}
Let $\Delta(V_{\bullet})$ be the closed convex hull of
%generated by 
$\bigcup_{m > 0} \frac{1}{m} \Gamma(V_{\bullet})_m$.
In the case where $V_m = H^0(T, mD)$ for all $m \geq 0$, $\Delta(V_{\bullet})$ is denoted by
$\Delta(D)$.
In the same arguments as \cite[Proposition~2.1]{LM} by using the above properties (1) and (2), 
we can see that
\[
\vol(\Delta(V_{\bullet})) = \lim_{m\to\infty} \frac{\dim_F V_m}{m^d}.
\]

\subsection{Arithmetic case}
Let $X$ be a $(d+1)$-dimensional, generically smooth, normal and projective arithmetic variety
(cf. Conventions and terminology~\ref{CT:arith:var}).
Let $X \to \Spec(O_K)$ be the Stein factorization of $X \to \Spec(\ZZ)$, so that
the generic fiber of $X \to \Spec(O_K)$ is geometrically irreducible.
Let $\overline{D} = (D, g)$ be an arithmetic $\RR$-Cartier divisor of $C^0$-type
(cf. Conventions and terminology~\ref{CT:arith:div}).
\begin{comment}
, that is,
$D = \sum_{i=1}^r a_i D_i$ is an element of $\Div(X)_{\RR} := \Div(X) \otimes_{\ZZ} \RR$ and
$g :  X(\CC) \to \RR \cup \{\pm\infty\}$ is a locally integrable
function such that
$g$ is invariant under the complex conjugation map $F_{\infty} : X(\CC) \to X(\CC)$ and,
for any point $x \in X(\CC)$, there are an open neighborhood $U_{x}$ of $x$ and
a continuous function $u_x$ on $U_x$ such that
\[
g = u_x + \sum_{i=1}^r a_i \log \vert f_i \vert^2\quad(a.e.)
\]
on $U_x$,
where $\Div(X)$ is the group of Cartier divisors on $X$ and
$f_i$ is a local equation of $D_i$ for each $i$ (for details, see \cite{MoArZariski}).
\end{comment}
We define $\aH(X, \overline{D})$ to be
\[
\aH(X, \overline{D}) := \widehat{\Gamma}^{\times}(X, \overline{D}) \cup \{ 0 \},
\]
where $\widehat{\Gamma}^{\times}(X, \overline{D}) := \left\{ \phi \in \Rat(X)^{\times} \mid
\overline{D} + \widehat{(\phi)} \geq (0,0) \right\}$ (for details, see Section~\ref{sec:asym:mult}).
Let $V_{\bullet}$ be a graded subalgebra of $\bigoplus_{m\geq 0} H^0(X_K, m D_K)$ over $K$.
Using $X$ and $\overline{D}$, we can define the natural filtration $\FF^{\star}_{\overline{D}}$ of $V_{\bullet}$
given by
\[
\FF_{\overline{D}}^t V_m = \langle V_m \cap \aH(X, m \overline{D} + (0, -2t)) \rangle_K
\]
for $t \in \RR$. Note that we use $(0, -2t)$ to ensure consistency with the notation in
\cite[Definition~2.3]{BC} .
It is easy to see that $\FF_{\overline{D}}^t V_m \cdot \FF_{\overline{D}}^{t'} V_{m'} \subseteq \FF_{\overline{D}}^{t+t'} V_{m+m'}$.
Thus, if we set 
\[
V_m^t = \FF_{\overline{D}}^{tm} V_m, 
\]
then $V^t_{\bullet} := \bigoplus_{m \geq 0} V_{m}^t$ forms a subalgebra of $V_{\bullet}$.
For each $m$, we define $e_{\min}(\overline{D}; V_m)$ and $e_{\max}(\overline{D}; V_m)$ to be
\[
\begin{cases}
e_{\min}(\overline{D}; V_m) := \inf \left\{ t \in \RR \mid \FF_{\overline{D}}^t V_m \not= V_m \right\},\\
e_{\max}(\overline{D}; V_m) := \sup \left\{ t \in \RR \mid \FF_{\overline{D}}^t V_m \not= \{ 0 \} \right\}.
\end{cases}
\]
Then, in the similar way as \cite[Section~2]{BC}, 
we can see the following:
\begin{enumerate}
\renewcommand{\labelenumi}{(\arabic{enumi})}
\item
$-\infty < e_{\min}(\overline{D}; V_m)$ for each $m$.

\item
There is a constant $C$ such that $e_{\max}(\overline{D}; V_m) \leq Cm$.

\item
We set $e_{\max}(\overline{D};V_{\bullet}) = \limsup_{m\to\infty} e_{\max}(\overline{D}; V_m)/m$.
If $V_{\bullet}$ contains an ample series, then
$V_{\bullet}^t$ also contains an ample series for $t < e_{\max}(\overline{D};V_{\bullet})$.
\end{enumerate}

We assume that $V_{\bullet}$ contains an ample series.
As in \cite[Definition~1.8]{BC},
we define $G_{(\overline{D};V_{\bullet})} : \Delta(V_{\bullet}) \to \RR \cup \{ -\infty \}$ and $\widehat{\Delta}(\overline{D};V_{\bullet})$ to be
\[
\begin{cases}
G_{(\overline{D};V_{\bullet})}(x) := \sup \left\{ t \in \RR \mid x \in \Delta(V^t_{\bullet}) \right\}, \\
\widehat{\Delta}(\overline{D};V_{\bullet}) := \left\{ (x, t) \in \Delta(V_{\bullet}) \times \RR \mid 0 \leq t \leq G_{(\overline{D};V_{\bullet})} \right\}.
\end{cases}
\]
Note that $G_{(\overline{D};V_{\bullet})} : \Delta(V_{\bullet}) \to \RR \cup \{ -\infty \}$ is an upper
semicontinuous concave function (cf. \cite[SubSection~1.3]{BC}).
In the case where $V_m = H^0(X_K, mD_K)$ for all $m \geq 0$,
$G_{(\overline{D};V_{\bullet})}$ and $\widehat{\Delta}(\overline{D};V_{\bullet})$ are denoted by
$G_{\overline{D}}$ and $\widehat{\Delta}(\overline{D})$ respectively.
Moreover, we define $\avol(\overline{D};V_{\bullet})$ to be
\[
\avol(\overline{D};V_{\bullet}) := \limsup_{m\to\infty} 
\frac{\# \log \left( V_m \cap \aH(X, m\overline{D}) \right)}{m^{d+1}/(d+1)!}.
\]
Then, in the similar way as  \cite[Theorem~2.8]{BC}, 
we have the following theorem:

\begin{Theorem}
\label{thm:integral:formula}
$\avol(\overline{D};V_{\bullet}) = (d+1)! [K : \QQ] \vol (\widehat{\Delta}(\overline{D};V_{\bullet}))$,
that is, 
\[
\avol(\overline{D};V_{\bullet}) = (d+1)! [K : \QQ] \int_{\Theta(\overline{D};V_{\bullet})}
G_{(\overline{D};V_{\bullet})}(x) dx,
\]
where
$\Theta(\overline{D};V_{\bullet})$ is the closure of  $\left\{ x \in \Delta(V_{\bullet}) \mid 
G_{(\overline{D};V_{\bullet})}(x) > 0 \right\}$.
\end{Theorem}

\renewcommand{\theTheorem}{\arabic{section}.\arabic{Theorem}}
\renewcommand{\theClaim}{\arabic{section}.\arabic{Theorem}.\arabic{Claim}}
\renewcommand{\theequation}{\arabic{section}.\arabic{Theorem}.\arabic{Claim}}

\section{Asymptotic multiplicity}
\label{sec:asym:mult}
Let $X$ be a $(d+1)$-dimensional, generically smooth, normal and projective arithmetic variety
(cf. Conventions and terminology~\ref{CT:arith:var}).
Let $\KK$ be either $\QQ$ or $\RR$. 
\begin{comment}
Let $\Div(X)$ be the group of Cartier divisors on $X$ and $\Div(X)_{\KK} := \Div(X) \otimes_{\ZZ} \KK$. 
A pair $\overline{D} = (D, g)$ 
%of an element $D = \sum_{i=1}^r a_i D_i$ of $\Div(X)_{\KK}$
%and a locally integrable function $g : X(\CC) \to \RR \cup \{\pm\infty\}$ 
is called an arithmetic $\KK$-Cartier divisor
of $C^{\infty}$-type (resp. of $C^{0}$-type) if
$D = \sum_{i=1}^r a_i D_i \in \Div(X)_{\KK}$,
$g : X(\CC) \to \RR \cup \{\pm\infty\}$ is a locally integrable function
%$g$ is 
invariant under the complex conjugation map $F_{\infty} : X(\CC) \to X(\CC)$ and,
for any point $x \in X(\CC)$, there are an open neighborhood $U_{x}$ of $x$ and
a $C^{\infty}$-function (resp. continuous function) $u_x$ on $U_x$ such that
\[
g = u_x + \sum_{i=1}^r a_i \log \vert f_i \vert^2\quad(a.e.)
\]
on $U_x$,
where $f_i$ is a local equation of $D_i$ for each $i$.
We denote the vector space consisting of arithmetic $\KK$-Cartier divisors of $C^{\infty}$-type
(resp. of $C^{0}$-type) by $\aDiv_{C^{\infty}}(X)_{\KK}$ (resp. $\aDiv_{C^{0}}(X)_{\KK}$).
\end{comment}
Let $\Rat(X)^{\times}_{\KK} := \Rat(X)^{\times} \otimes_{\ZZ} \KK$, and
let 
\[
(\ )_{\KK} : \Rat(X)^{\times}_{\KK} \to \Div(X)_{\KK}\quad\text{and}\quad
\widehat{(\ )}_{\KK} : \Rat(X)^{\times}_{\KK} \to \aDiv_{C^{\infty}}(X)_{\KK}
\]
 be the natural extensions of
the homomorphisms 
\[
\Rat(X)^{\times} \to \Div(X)\quad\text{and}\quad
\Rat(X)^{\times} \to \aDiv_{C^{\infty}}(X)
\]
given by
$\phi \mapsto (\phi)$ and 
$\phi \mapsto \widehat{(\phi)}$ respectively.
Let $\overline{D}$ be an arithmetic $\RR$-Cartier divisor of $C^0$-type
(cf. Conventions and terminology~\ref{CT:arith:div}).
%(for details, see \cite{MoArZariski}).
We define $\widehat{\Gamma}^{\times}(X, \overline{D})$ and $\widehat{\Gamma}^{\times}_{\KK}(X, \overline{D})$
to be
\[
\begin{cases}
\widehat{\Gamma}^{\times}(X, \overline{D}) := \left\{ \phi \in \Rat(X)^{\times} \mid
\overline{D} + \widehat{(\phi)} \geq (0,0) \right\}, \\
\widehat{\Gamma}^{\times}_{\KK}(X, \overline{D}) := \left\{ \phi \in \Rat(X)^{\times}_{\KK} \mid
\overline{D} + \widehat{(\phi)}_{\KK} \geq (0,0) \right\}.
\end{cases}
\]
Note that $\widehat{\Gamma}^{\times}_{\QQ}(X, \overline{D}) = \bigcup_{n=1}^{\infty} \widehat{\Gamma}^{\times}(X, n\overline{D})^{1/n}$.
Moreover, $\aH(X, \overline{D})$ and $\aH_{\KK}(X, \overline{D})$ are defined by
\[
\aH(X, \overline{D}) := \widehat{\Gamma}^{\times}(X, \overline{D}) \cup \{ 0 \}
\quad\text{and}\quad
\aH_{\KK}(X, \overline{D}) := \widehat{\Gamma}^{\times}_{\KK}(X, \overline{D}) \cup \{ 0 \}.
\]
For $\xi \in X$,
we define the {\em $\KK$-asymptotic multiplicity of $\overline{D}$ at $\xi$} to be
\[
\mu_{\KK,\xi}(\overline{D}) :=
\begin{cases}
\inf \left\{ \mult_{\xi}(D + (\phi)_{\KK}) \mid \phi \in \Gamma_{\KK}^{\times}(X, \overline{D}) \right\} &
\text{if $\Gamma_{\KK}^{\times}(X, \overline{D}) \not= \emptyset$}, \\
\infty & \text{otherwise}.
\end{cases}
\]

First let us observe the elementary properties of the $\KK$-asymptotic multiplicity
(cf. \cite[Proposition~6.5.2 and Proposition~6.5.3]{MoArZariski}).

\begin{Proposition}
\label{prop:mu:basic}
Let $\overline{D}$ and $\overline{E}$ be arithmetic $\RR$-Cartier divisors of $C^0$-type.
Then we have the following:
\begin{enumerate}
\renewcommand{\labelenumi}{(\arabic{enumi})}
\item
$\mu_{\KK,\xi}(\overline{D} + \overline{E}) \leq 
\mu_{\KK,\xi}(\overline{D}) + \mu_{\KK,\xi}(\overline{E})$.

\item
If $\overline{D} \leq \overline{E}$, then 
$\mu_{\KK,\xi}(\overline{E}) \leq \mu_{\KK,\xi}(\overline{D}) + \mult_{\xi}(E - D)$.

\item
$\mu_{\KK,\xi}(\overline{D} + \widehat{(\phi)}_{\KK}) = \mu_{\KK,\xi}(\overline{D})$ for 
$\phi \in \Rat(X)^{\times}_{\KK}$.

\item
$\mu_{\KK,\xi}(a\overline{D}) = a \mu_{\KK,\xi}(\overline{D})$ for $a \in \KK_{>0}$.

\item
$0 \leq \mu_{\RR,\xi}(\overline{D}) \leq \mu_{\QQ,\xi}(\overline{D})$.

\item
If $\overline{D}$ is nef and big, then $\mu_{\KK,\xi}(\overline{D}) = 0$.
\end{enumerate}
\end{Proposition}

\begin{proof}
(1)  
If $\widehat{\Gamma}^{\times}_{\KK}(X, \overline{D}+\overline{E}) = \emptyset$,
then either $\widehat{\Gamma}^{\times}_{\KK}(X, \overline{D}) = \emptyset$ or
$\widehat{\Gamma}^{\times}_{\KK}(X, \overline{E}) = \emptyset$, so that
we may assume that $\widehat{\Gamma}^{\times}_{\KK}(X, \overline{D}+\overline{E}) \not= \emptyset$.
Thus we may also assume that $\widehat{\Gamma}^{\times}_{\KK}(X, \overline{D}) \not= \emptyset$ and
$\widehat{\Gamma}^{\times}_{\KK}(X, \overline{E}) \not= \emptyset$.
Therefore, the assertion follows because $\phi \psi \in \widehat{\Gamma}^{\times}_{\KK}(X, \overline{D}+\overline{E})$ for all $\phi \in \widehat{\Gamma}^{\times}_{\KK}(X, \overline{D})$ and
$\psi \in \widehat{\Gamma}^{\times}_{\KK}(X, \overline{E})$.

(2) is derived from (1).

(3) 
The assertion follows from the following:
\[
\psi \in \widehat{\Gamma}^{\times}_{\KK}(X, \overline{D})\quad\Longleftrightarrow \quad
\psi\phi^{-1} \in \widehat{\Gamma}^{\times}_{\KK}(X, \overline{D}+ \widehat{(\phi)}_{\KK}).
\]

(4) Note that $\psi \in \widehat{\Gamma}^{\times}_{\KK}(X, \overline{D})$ if and only if
$\psi^a \in \widehat{\Gamma}^{\times}_{\KK}(X, a\overline{D})$, and that
\[
\mult_{\xi}(a D + (\psi^a)_{\KK}) = a \mult_{\xi}(D + (\psi)_{\KK}),
\]
which implies (4).

(5) is obvious.

(6) follows from (5) and \cite[Proposition~6.5.3]{MoArZariski}
\end{proof}

\begin{Remark}
Theorem~\ref{thm:equal:mu:Q:mu:R} says that if $\overline{D}$ is big,
then $\mu_{\RR, \xi}(\overline{D}) = \mu_{\QQ, \xi}(\overline{D})$.
In general, it does not hold.
Let $\PP^1_{\ZZ} = \Proj(\ZZ[T_0, T_1])$ be the projective line over $\ZZ$.
We set $D := \{ T_0 = 0 \}$ and $z := T_1/T_0$.
Let $a_0, a_1 \in \RR_{>0}$ such that $a_0 + a_1 = 1$ and $a_0 \not\in \QQ$.
Let $\overline{D}$ be an arithmetic divisor of $C^{\infty}$-type on $\PP^1_{\ZZ}$ 
given by
\[
\overline{D} := \left( D, \log (a_0 + a_1 \vert z \vert^2)\right).
\]
Then it is easy to see that
\[
\Gamma_{\QQ}^{\times}(X, \overline{D}) = \emptyset\quad\text{and}\quad
\Gamma_{\RR}^{\times}(X, \overline{D}) \ni z^{a_1}
\]
(for details, see \cite[(6) in Theorem~2.3]{MoBig}).
Thus $\mu_{\QQ, \xi}(\overline{D}) = \infty$ for all $\xi \in \PP^1(\overline{\QQ})$ and
\[
\mu_{\RR, \xi}(\overline{D}) 
\begin{cases}
\leq a_0 & \text{if $\xi = (0:1)$}, \\
\leq a_1 & \text{if $\xi = (1:0)$}, \\
= 0 & \text{if $\xi \in \PP^1(\overline{\QQ}) \setminus \{ (0:1), (1:0) \}$}.
\end{cases}
\]
\end{Remark}

Next we consider the following lemmas, which will be important for the proof of Theorem~\ref{thm:equal:mu:Q:mu:R}.

\begin{Lemma}
\label{lem:cont:asym:mult}
We assume that $\overline{D}$ is big.
Let $a = \inf\{ x \in \RR \mid \avol(\overline{D} + (0, x)) > 0 \}$ and  
let
$f : (a, \infty) \to \RR$ be the function given by
$f(x) = \mu_{\KK,\xi}(\overline{D} + (0, x))$.
Then $f$ is a monotone decreasing continuous function.
\end{Lemma}

\begin{proof}
For $x, y \in (a, \infty)$ with $x \leq y$, we have $\overline{D} + (0, x) \leq \overline{D} + (0, y)$, and
hence $f(x) \geq f(y)$ by (2) in Proposition~\ref{prop:mu:basic}.
Here let us see that $f$ is a $\KK$-convex function on $(-\infty, a) \cap \KK$, that is,
\[
f(tx + (1-t)y) \leq t f(x) + (1-t)f(y)
\]
holds for all $x, y \in (a, \infty) \cap \KK$ and
$t \in [0,1] \cap \KK$. Indeed, by using (1) and (4) in Proposition~\ref{prop:mu:basic},
\begin{align*}
f(tx + (1-t)y) & = \mu_{\KK,\xi}\left(t(\overline{D} + (0, x)) + (1-t)(\overline{D} + (0, y))\right) \\
& \leq  \mu_{\KK,\xi}\left(t(\overline{D} + (0, x))\right) + \mu_{\KK,\xi}\left((1-t)(\overline{D} + (0, y))\right) \\
& = t f(x) + (1-t)f(y).
\end{align*}
The continuity of an $\RR$-convex function on an open interval of $\RR$ is well-known
(cf. \cite[Theorem~5.5.1]{Web}), 
%Thus the continuity of $f$ is obvious in the case where $\KK = \RR$ (cf. \cite[Theorem~5.5.1]{Web}), 
so that
we assume $\KK = \QQ$.
We check the continuity of $f$ at $x \in (a, \infty)$.
By \cite[Proposition~1.3.1]{MoArLin},
there are positive numbers $\epsilon$ and $L$ such that $(x - \epsilon, x + \epsilon) \subseteq (a, \infty)$ and
\[
0 \leq f(v) - f(u) \leq  L (u - v)
\]
for all $u, v \in (x - \epsilon, x + \epsilon) \cap \QQ$ with $u \geq v$.
Let $y, z \in (x - \epsilon, x + \epsilon)$ with $y \geq z$.
Here we choose arbitrary rational numbers $u, v$ such that
$x -  \epsilon < v \leq z \leq y \leq u < x + \epsilon$. Then
\[
0 \leq f(z) - f(y) \leq f(v) - f(u) \leq L(u - v),
\]
and hence $0 \leq f(z) - f(y) \leq L (y - z)$ holds.
Therefore, the lemma follows.
\end{proof}

\begin{Lemma}
\label{lem:Q:approximation}
We assume that $\overline{D}$ is effective.
Let $\phi_1, \ldots, \phi_r \in \Rat(X)^{\times}_{\QQ}$ and
$a_1, \ldots, a_r \in \RR$ with
$a_1 \widehat{(\phi_1)} + \cdots + a_r \widehat{(\phi_r)} + \overline{D} \geq 0$.
Then there is a subspace $W$ of $\QQ^r$ over $\QQ$ with the following properties:
\begin{enumerate}
\renewcommand{\labelenumi}{(\arabic{enumi})}
\item
$\dim_{\QQ} W = \dim_{\QQ} \langle a_1, \ldots, a_r \rangle_{\QQ}$,
where $\langle a_1, \ldots, a_r \rangle_{\QQ}$ is the subspace of $\RR$ generated by
$a_1, \ldots, a_r$ over $\QQ$.

\item
$(a_1, \ldots, a_r) \in W_{\RR} := W \otimes_{\QQ} \RR$.

\item
For positive numbers $\epsilon$ and $\epsilon'$, there is a positive number $\delta$ such that
\[
c_1 \widehat{(\phi_1)} + \cdots + c_r \widehat{(\phi_r)} + \overline{D} + (0, \epsilon') \geq 0
\]
for any $(c_1, \ldots, c_r) \in W_{\RR}$ with 
\[
\Vert (c_1, \ldots, c_r) - (a_1/(1+\epsilon),
\ldots, a_r/(1+\epsilon)) \Vert \leq \delta,
\]
where $\Vert\cdot\Vert$ is the standard $L^2$-norm
on $\RR^r$.
\end{enumerate}
\end{Lemma}

\begin{proof}
First we assume that $a_1, \ldots, a_r$ are linearly independent over $\QQ$, that is,
\[
\dim_{\QQ} \langle a_1, \ldots, a_r \rangle_{\QQ} = r.
\]
Replacing $\phi_1, \ldots, \phi_r, a_1, \ldots, a_r$ by 
$\phi_1^n, \ldots, \phi_r^n, a_1/n, \ldots, a_r/n$ respectively
for some $n \in \ZZ_{>0}$,
we may assume that $\phi_1, \ldots, \phi_r \in \Rat(X)^{\times}$.
The set of all prime divisors on $X$ is denoted by $I$.
Moreover, for $x \in X(\CC)$, the set $\{ B \in I \mid x \in B(\CC) \}$ is denoted by $I_x$.
For $B \in I_x$,
let $B(\CC)_x = B_1 + \cdots + B_{n_{B_x}}$ be the irreducible decomposition of $B(\CC)$ at $x$ on $X(\CC)$,
that is, $B_1, \dots, B_{n_{B_x}}$ are irreducible components of $B(\CC)$ on $X(\CC)$ passing through $x$.
Note that $\ord_{B}(\phi) = \ord_{B_j}(\phi)$ for $\phi \in \Rat(X)^{\times}$ and 
$j=1, \ldots, n_{B_x}$.
We set $\overline{D} = (D, g)$ and $D = \sum_{B \in I} d_B B$.
By our assumption, $d_B \geq 0$ for all $B \in I$ and $g \geq 0$.
For $\pmb{c} = (c_1, \ldots, c_r) \in \RR^r$,
we define $\phi^{\pmb{c}}$,
$D_{\pmb{c}}$, $g_{\pmb{c}}$ and
$\overline{D}_{\pmb{c}}$ to be
\[
\begin{cases}
\phi^{\pmb{c}} := \phi_1^{c_1} \cdots \phi_r^{c_r}, \\
D_{\pmb{c}} := (\phi^{\pmb{c}})_{\RR} + D = \sum_{i=1}^r c_i (\phi_i) + D, \\
g_{\pmb{c}} := \sum_{i=1}^r (-c_i) \log \vert \phi_i \vert^2 + g, \\
\overline{D}_{\pmb{c}} := (D_{\pmb{c}}, g_{\pmb{c}}) = \widehat{(\phi^{\pmb{c}})}_{\RR} + \overline{D}.
\end{cases}
\]
Note that 
\[
D_{\pmb{c}} = \sum_{B \in I} (\ord_B(\phi^{\pmb{c}}) + d_B) B,
\]
where  
$\ord_B : \Rat(X)^{\times}_{\RR} \to \RR$
is the natural extension of the homomorphism $\Rat(X)^{\times} \to \ZZ$ given by $\psi \mapsto \ord_B(\psi)$.
Around $x \in X(\CC)$, we set 
\[
\phi_i = \rho_i \prod_{B \in I_x} \prod_{j=1}^{n_{B_x}} f_{B_j}^{\ord_B(\phi_i)},
\]
where $f_{B_j}$ is a local equation of $B_j$ around $x$ and $\rho_i \in \OO_{X(\CC), x}^{\times}$.
Then
\[
\phi^{\pmb{c}} =  
\rho_1^{c_1} \cdots \rho_r^{c_r} \prod_{B \in I_x} \prod_{j=1}^{n_{B_x}} 
f_{B_j}^{\ord_B(\phi^{\pmb{c}})}.
\]
Thus, if we set
\[
g = \sum_{B \in I_x}  \sum_{j=1}^{n_{B_x}} -d_B \log \vert f_{B_j} \vert^2 + u_x
\]
around $x$, then
\addtocounter{Claim}{1}
\begin{align}
g_{\pmb{c}} & = g + \sum_{B \in I_x}  \sum_{j=1}^{n_{B_x}} (-\ord_B(\phi^{\pmb{c}}))
\log \vert f_{B_j} \vert^2 + \sum_{i=1}^r (-c_i) \log \vert \rho_i \vert^2  \notag \\
& = \sum_{B \in I_x}  \sum_{j=1}^{n_{B_x}} -(d_B+\ord_B(\phi^{\pmb{c}})) 
\log \vert f_{B_j} \vert^2 + \sum_{i=1}^r (-c_i) \log \vert \rho_i \vert^2 + u_x.
\label{eqn:lem:Q:approximation:1}
\end{align}
We put $S = \bigcup_{i=1}^r \Supp((\phi_i))$ and $\pmb{a} = (a_1, \ldots, a_r)$.

\begin{Claim}
\label{claim:thm:equal:mu:Q:mu:R:1}
\begin{enumerate}
\renewcommand{\labelenumi}{(\roman{enumi})}
\item
$\widehat{(\phi^{\pmb{a}/(1+\epsilon)})}_{\RR} + \overline{D} \geq 0$.
In particular, $g_{\pmb{a}/(1 + \epsilon)} \geq 0$.

\item 
$\ord_B(\phi^{\pmb{a}/(1 + \epsilon)}) + d_B > 0$
for all $B \in I$ with $B \subseteq S$. In particular, we can find $\delta_0 > 0$ such that
$(\phi^{\pmb{c}})_{\RR} + D \geq 0$ for any $\pmb{c} \in \RR^r$ with 
$\Vert \pmb{c} - \pmb{a}/(1+\epsilon) \Vert \leq \delta_0$.
\end{enumerate}
\end{Claim}

\begin{proof}
(i) The assertion follows from the following:
\[
\widehat{(\phi^{\pmb{a}/(1+\epsilon)})}_{\RR} + \overline{D} = 
\frac{1}{1+\epsilon}(\widehat{(\phi^{\pmb{a}})}_{\RR} + \overline{D})  
+ \left(1 - \frac{1}{1+\epsilon}\right) \overline{D}.
\]

(ii) It is sufficient to show that $\ord_{B}(\phi^{\pmb{a}}) + (1 + \epsilon) d_{B} > 0$
for all $B \in I$ with $B \subseteq S$.
First of all, note that $\ord_{B}(\phi^{\pmb{a}}) + d_{B} \geq 0$. 
If either $\ord_{B}(\phi^{\pmb{a}}) > 0$ or $d_{B} > 0$, then the assertion is obvious, so that we assume 
$\ord_{B}(\phi^{\pmb{a}}) \leq 0$ and $d_{B} = 0$.
Then 
\[
\ord_{B}(\phi^{\pmb{a}}) = a_1 \ord_{B}(\phi_1) + \cdots + a_r \ord_{B}(\phi_r) = 0,
\]
which yields $\ord_{B}(\phi_1) = \cdots = \ord_{B}(\phi_r) = 0$
by the linear independency of $a_1, \ldots, a_r$ over $\QQ$.
This is a contradiction because $B \subseteq S$.
\end{proof}

\begin{Claim}
For each $x \in X(\CC)$, there are $\delta_x >0$ 
and an open neighborhood $U_x$ of $x$
such that $g_{\pmb{c}} + \epsilon' \geq 0$ on $U_x$ for any
$\pmb{c} \in \RR^r$ with
$\Vert \pmb{c} - \pmb{a}/(1 + \epsilon) \Vert \leq \delta_x$.
\end{Claim}

\begin{proof}
First we assume $x \in S(\CC)$.
For $B \in I$ with $B \subseteq S$, we set 
\[
d'_B = \ord_B(\phi^{\pmb{a}/(1 + \epsilon)}) + d_B > 0.
\]
We choose $\delta' > 0$ such that
\[
\frac{1}{2} d'_B \leq d_B+\ord_B(\phi^{\pmb{c}}) \leq \frac{3}{2} d'_B
\]
for all
$\pmb{c} \in \RR^r$ and $B \in I$ with
$\Vert \pmb{c} - \pmb{a}/(1 + \epsilon) \Vert \leq \delta'$ and $B \subseteq S$.
Note that there are an open neighborhood $U_x$ and a constant $M$ such that
\[
\sum_{i=1}^r (-c_i) \log \vert \rho_i \vert^2 + u_x \geq M
\]
over $U_x$ for all $\pmb{c} \in \RR^r$ with
$\Vert \pmb{c} - \pmb{a}/(1 + \epsilon) \Vert \leq \delta'$.
Moreover, shrinking $U_x$ if necessarily, we may assume that
$\vert f_{B_j} \vert \leq 1$ for all $B_j$ with either $B \subseteq S$ or $d_B > 0$
because the set $\{ B \in I \mid \text{$B \subseteq S$ or $d_B > 0$} \}$ is finite and
$f_{B_j}(x) = 0$.
Thus, by using \eqref{eqn:lem:Q:approximation:1},
\begin{align*}
g_{\pmb{c}} &\geq  \sum_{B \in I_x}  \sum_{j=1}^{n_{B_x}} -(d_B+\ord_B(\phi^{\pmb{c}})) 
\log \vert f_{B_j} \vert^2 + M \\
& = \sum_{\substack{B \in I_x\\ B\subseteq S}}  
\sum_{j=1}^{n_{B_x}} -(d_B+\ord_B(\phi^{\pmb{c}})) 
\log \vert f_{B_j} \vert^2  + \sum_{\substack{B \in I_x\\ B \not\subseteq S, d_B > 0}}  
\sum_{j=1}^{n_{B_x}} -d_B 
\log \vert f_{B_j} \vert^2 + M \\
& \geq  \sum_{B \in I_x,B\subseteq S}  
\sum_{j=1}^{n_{B_x}} -\frac{1}{2} d'_B
\log \vert f_{B_j} \vert^2  + M.
\end{align*}
Note that $\lim_{y\to x} (-d'_B)
\log \vert f_{B_j}(y) \vert^2 = \infty$. Thus, the assertion follows if we take a smaller neighborhood
$U_x$.

\medskip
Next we consider the case where $x \not\in S(\CC)$. Then, by (i) in Claim~\ref{claim:thm:equal:mu:Q:mu:R:1},
\begin{align*}
g_{\pmb{c}} + \epsilon' & = g + \sum_{i=1}^r (-c_i) \log \vert \rho_i \vert^2 + \epsilon' =
g_{\pmb{a}/(1+\epsilon)} + \epsilon' + \sum_{i=1}^r (a_i/(1+\epsilon) -c_i) \log \vert \rho_i \vert^2 \\
& \geq \epsilon' + \sum_{i=1}^r (a_i/(1+\epsilon) -c_i) \log \vert \rho_i \vert^2.
\end{align*}
Thus the assertion follows.
\end{proof}

As $X(\CC) = \bigcup_{x \in X(\CC)} U_x$ and $X(\CC)$ is compact, there are $x_1, \ldots, x_l \in
X(\CC)$ such that $X(\CC) = U_{x_1} \cup \cdots \cup U_{x_{l}}$.
Therefore, if we set 
$\delta_1 = \{ \delta_{x_1}, \ldots, \delta_{x_l} \}$, then
\[
g_{\pmb{c}} + \epsilon'  \geq 0
\]
for all $\pmb{c} \in \RR^r$ with
$\Vert \pmb{c} - \pmb{a}/(1 + \epsilon) \Vert \leq \delta_1$, and hence, 
if we put $\delta = \min \{ \delta_0, \delta_1 \}$, then,
by (ii) in Claim~\ref{claim:thm:equal:mu:Q:mu:R:1}, 
\[
\overline{D}_{\pmb{c}} + (0, \epsilon') \geq 0
\]
for all $\pmb{c} \in \RR^r$ with
$\Vert \pmb{c} - \pmb{a}/(1 + \epsilon) \Vert \leq \delta$.

\bigskip
Finally we consider the lemma without the linear independency of
$a_1, \ldots, a_r$ over $\QQ$.
We set $s = \dim_{\QQ} \langle a_1, \ldots, a_r \rangle_{\QQ}$.
If $s= 0$ (i.e. $a_1 = \cdots = a_r = 0$), then we can take $W$ as $\{ (0, \ldots, 0) \}$, so that
we may assume $s \geq 1$.
Renumbering $a_1, \ldots, a_r$,  we may further assume that $a_1, \ldots, a_s$ are linearly independent.
We set $a_i = \sum_{j=1}^s e_{ij} a_j$ ($i=1, \ldots, r$) and $\psi_j = \prod_{i=1}^r \phi_i^{e_{ij}}$ ($j=1, \ldots, s$).
Note that $e_{ij} \in \QQ$, and hence $\psi_j \in \Rat(X)^{\times}_{\QQ}$.
Let $\alpha : \RR^s \to \RR^r$ be the homomorphism given by
\[
\alpha(x_1, \ldots, x_s) = (\alpha_1(x_1, \ldots, x_s), \ldots, \alpha_r(x_1, \ldots, x_s))\quad\text{and}\quad
\alpha_i(x_1, \ldots, x_s) = \sum_{j=1}^s e_{ij} x_j.
\]
As the rank of $(e_{ij})$ is $s$, $\alpha$ is injective. In addition,
$(a_1, \ldots, a_r) = \alpha(a_1, \ldots, a_s)$ and
\[
x_1 \widehat{(\psi_1)} + \cdots + x_s \widehat{(\psi_s)} =
\alpha_1(x_1, \ldots, x_s) \widehat{(\phi_1)} + \cdots + \alpha_r(x_1, \ldots, x_s) \widehat{(\phi_r)}.
\]
for $(x_1, \ldots, x_s) \in \RR^s$.
Therefore, if we put $W = \alpha(\QQ^s) \subseteq \QQ^r$,
then the assertion follows from the previous observation.
\end{proof}

The following theorem is the main result of this section.

\begin{Theorem}
\label{thm:equal:mu:Q:mu:R}
If $\overline{D}$ is big, then $\mu_{\QQ,\xi}(\overline{D}) = \mu_{\RR,\xi}(\overline{D})$.
\end{Theorem}

\begin{proof}
First of all, by (3) in Proposition~\ref{prop:mu:basic},
we may assume that $\overline{D}$ is effective.
Moreover, by (5) in Proposition~\ref{prop:mu:basic},  
$\mu_{\RR,\xi}(\overline{D}) \leq \mu_{\QQ,\xi}(\overline{D})$, so that
we consider the converse inequality. For this purpose, it is sufficient to show that
\[
\mu_{\QQ,\xi}(\overline{D}) \leq \mult_{\xi}(D + (\psi)_{\RR})
\]
for all $\psi \in \widehat{\Gamma}_{\RR}^{\times}(X, \overline{D})$.
We choose $\phi_1, \ldots, \phi_r \in \Rat(X)^{\times}$ and
$\pmb{a} = (a_1, \ldots, a_r) \in \RR^r$ such that $a_1, \ldots, a_r$ are linearly independent over $\QQ$ and
$\psi = \phi_1^{a_1} \cdots \phi_r^{a_r}$.
Let $\epsilon$ be a positive number. 
Applying Lemma~\ref{lem:Q:approximation} to the case $\epsilon = \epsilon'$, 
we can find a sequence $\{ \pmb{c}_n \}_{n=1}^{\infty}$ 
in $\QQ^r$ such that $\lim_{n\to\infty} \pmb{c}_n = \pmb{a}/(1 + \epsilon)$ and 
$\phi^{\pmb{c}_n} \in \widehat{\Gamma}^{\times}_{\QQ}(X, \overline{D}+(0,\epsilon))$ for all $n$.
Thus we have
\[
\mu_{\QQ,\xi}(\overline{D} + (0,\epsilon)) \leq \mult_{\xi}(D_{\pmb{c}_n})
\]
for all $n$, and hence, $\mu_{\QQ,\xi}(\overline{D} + (0,\epsilon)) \leq \mult_{\xi}(D_{\pmb{a}/(1 + \epsilon)})$.
Therefore, by Lemma~\ref{lem:cont:asym:mult},
\[
\mu_{\QQ,\xi}(\overline{D}) =\lim_{\epsilon\downarrow 0} \mu_{\QQ,\xi}(\overline{D} + (0,\epsilon)) \leq \lim_{\epsilon\downarrow 0} \mult_{\xi}(D_{\pmb{a}/(1 + \epsilon)}) 
= \mult_{\xi}(D + (\psi)_{\RR}).
\]
\end{proof}

\section{Proof of Theorem~\ref{thm:vol:base:cond}}
\renewcommand{\theClaim}{\arabic{section}.\arabic{Claim}}
\renewcommand{\theequation}{\arabic{section}.\arabic{Claim}}
\setcounter{Claim}{0}

In this section, we give the proof of Theorem~\ref{thm:vol:base:cond}.
Since 
\[
\avol(\overline{D} ; \mu_1 \xi_1, \ldots,  \mu_l \xi_l) \leq \avol(\overline{D} ; \mu_i \xi_i),
\]
it is sufficient to show the following:
\addtocounter{Claim}{1}
\begin{equation}
\label{eqn:thm:vol:base:cond:1}
\text{If $\overline{D}$ is big and $\mu > \mu_{\RR,\xi}(\overline{D})$ for $\xi \in X_{\QQ}$,
then $\avol(\overline{D} ; \mu \xi)  < \avol(\overline{D})$.}
\end{equation}

\bigskip
Let $B$ be the Zariski closure of $\{ \xi \}$ in $X$.
Let us begin with the following claim:

\begin{Claim}
\label{claim:thm:vol:base:cond:2}
We may assume that $B$ is a prime divisor.
\end{Claim}

\begin{proof}
Let $\nu_{B} : X_B \to X$ be the blowing-up along $B$.
As $X_{\QQ}$ is regular,  we can find a unique prime divisor $E_B$ on
$X_B$ such that $\nu_B(E_B) = B$. Let $\nu' : X' \to X_B$ be 
a projective birational morphism such that $X'$ is normal and $\nu$ yields 
a resolution of singularities on the generic fiber.
%such that $X'$ is normal.
Let $B'$ be the strict transform of $E_B$ and let $\nu : X' \to X$ be the composition
of $\nu' : X' \to X_B$ and $\nu_B : X_B \to X$. If $\xi'$ is the generic point of $B'$,
then it is easy to see that $\ord_\xi(f) = \ord_{\xi'}(\nu^*(f))$ for all $f \in \Rat(X)^{\times}$, and
hence $\mult_{\xi}(L) = \mult_{\xi'}(\nu^*(L))$ for all $L \in \Div(X)_{\RR}$.
Moreover, the natural homomorphism $\nu^* : \Rat(X) \to \Rat(X')$ yields a bijection
$\aH(X, n\overline{D}) \to \aH(X', \nu^*(n\overline{D}))$. Therefore, we have
\[
\#\aH(X, n\overline{D}; n\mu \xi) = \#\aH(X', n\nu^*(\overline{D}); n\mu \xi'),
\]
which implies $\avol(\overline{D}; \mu \xi) = \avol(\nu^*(\overline{D}); \mu \xi')$,
as required.
\end{proof}

From now on, we assume that $B$ is a prime divisor.
Let $\mu_0 = \mu_{\RR,\xi}(\overline{D})$ and let $X \to \Spec(O_K)$ be the Stein factorization of $X \to \Spec(\ZZ)$.

\begin{Claim}
\label{claim:thm:vol:base:cond:3}
There is a positive number $\epsilon_0$ such that $D - (\mu_0 + \epsilon) B$ is big on $X_K$ for all $0 \leq \epsilon \leq \epsilon_0$.
\end{Claim}

\begin{proof}
Let $\overline{A}$ be a big arithmetic Cartier divisor of $C^0$-type
on $X$ such that $\overline{A} \geq (0,0)$ and $B \not\subseteq \Supp(A)$.
We can choose a sufficiently small positive number $a$ such that $\avol(\overline{D} - a \overline{A}) > 0$.
In particular, there is $\phi \in \Rat(X)^{\times}_{\QQ}$ such that $\overline{D} - a \overline{A} + \widehat{(\phi)}_{\QQ}
\geq 0$.
By (2) in Proposition~\ref{prop:mu:basic},
\[
\mu_0 = \mu_{\RR,\xi}(\overline{D}) \leq \mu_{\RR,\xi}(\overline{D} - a \overline{A}) + \mult_{\xi}(a A) = 
\mu_{\RR,\xi}(\overline{D} - a \overline{A}) \leq \mult_{\xi}(D - aA + (\phi)_{\QQ}).
\]
Thus $D - a A + (\phi)_{\QQ} \geq \mu_0 B$, and hence $D - \mu_0 B \geq aA - (\phi)_{\QQ}$.
In particular, $D - \mu_0B$ is big on $X_K$, and hence the assertion follows.
\end{proof}

It is sufficient to show \eqref{eqn:thm:vol:base:cond:1} in the case where $\mu = \mu_0 + \epsilon$ with
$0 < \epsilon < \epsilon_0$.
We set $V_m = H^0(X_K, mD_K - m\mu B_K)$.
Note that $\avol(\overline{D};V_{\bullet}) = \avol(\overline{D} ; \mu \xi)$.
By Claim~\ref{claim:thm:vol:base:cond:3}, $V_{\bullet}$ contains an ample series.

We choose $P \in X(\overline{K})$ and a local system of parameters $z_P = (z_1, \ldots, z_d)$ at $P$
such that $P$ is a regular point of $B_K$ and $z_1$ is a local equation of $B$ at $P$.

\begin{Claim}
\label{claim:thm:vol:base:cond:4}
If we set $\mult_{z_P}(L) = (x_1, \ldots, x_d)$ for $L \in \Div(X)_{\RR}$, then
$x_1 = \mult_{\xi}(L)$.
\end{Claim}

\begin{proof}
First we assume that $L \in \Div(X)$ and $L$ is effective.
Let $f$ be a local equation of $L$ around $P$.
We set $f = \sum_{i=a}^{\infty} f_i z_1^{i}$ in $\overline{K}[\![z_1, \ldots, z_d]\!]$,
where $a \in \ZZ_{\geq 0}$, $f_i \in K[\![z_2, \ldots, z_d]\!]$ and
$f_a \not= 0$. Then $a = \mult_{\xi}(f)$. Moreover, the lowest term with respect to the lexicographical order must
appear in $f_a z_1^{a}$. Thus $x_1 = a$, as required.

In general, we set $L = \sum_{i=1}^l a_i L_i$, where $a_1, \ldots, a_l \in \RR$ and
$L_i$'s are effective divisors. Moreover, if we set $\mult_{z_P}(L_i) = (x_{i1},\ldots, x_{id})$,
then $x_{i1} = \mult_{\xi}(L_i)$ by the previous observation.
On the other hand, as
\[
\mult_{z_P}(L) = \sum_{i=1}^l a_i \mult_{z_P}(L_i),
\]
the first entry of $\mult_{z_P}(L)$ is equal to 
\[
\sum_{i=1}^l a_i x_{i1} = \sum_{i=1}^l a_i \mult_{\xi}(L_i) = \mult_{\xi}(L),
\]
as desired.
\end{proof}

By Claim~\ref{claim:thm:vol:base:cond:4},
$\Delta(V_{\bullet}) \subseteq \Delta(D_K) \cap \{ x_1 \geq \mu \}$
and $G_{\overline{D}} \geq G_{(\overline{D};V_{\bullet})}$.
As in Theorem~\ref{thm:integral:formula},
let 
$\Theta(\overline{D})$ and $\Theta(\overline{D};V_{\bullet})$ be the closures of
\[
\left\{ x \in \Delta(D_K) \mid G_{\overline{D}}(x) > 0 \right\}\quad\text{and}\quad
\left\{ x \in \Delta(V_{\bullet}) \mid G_{(\overline{D};V_{\bullet})}(x) > 0 \right\}
\]
respectively.
Clearly $\Theta(\overline{D};V_{\bullet}) \subseteq \Theta(\overline{D}) \cap \{ x_1 \geq \mu \}$.

\begin{Claim}
\label{claim:thm:vol:base:cond:5}
$\Theta(\overline{D}) \cap \{ x_1 < \mu \} \not= \emptyset$.
\end{Claim}

\begin{proof}
As $\overline{D}$ is big,
$\mu_0 = \mu_{\QQ,\xi}(\overline{D})$ by Theorem~\ref{thm:equal:mu:Q:mu:R}, and hence,
by Lemma~\ref{lem:cont:asym:mult},
there is a positive number $t_0$ such that 
\[
\mu_0 \leq \mu_{\QQ,\xi}(\overline{D} + (0, -2t_0)) < \mu_0 + \epsilon/2.
\]
Thus we can find $\phi \in \widehat{\Gamma}^{\times}_{\QQ}(X, \overline{D}+ (0, -2t_0))$
such that 
\[
\mu_{\QQ,\xi}(\overline{D} + (0, -2t_0)) \leq \mult_{\xi}(D + (\phi)_{\QQ}) \leq \mu_0 + \epsilon/2 < \mu.
\]
Moreover, as  $\phi \in \widehat{\Gamma}^{\times}_{\QQ}(X, \overline{D}+ (0, -2t_0))$,
if we set $x = \mult_{z_P}(D + (\phi)_{\QQ})$, then $G_{\overline{D}}(x) > 0$, and hence
$x \in \Theta(\overline{D})$.
\end{proof}

Here we fix notation.
Let $T$ be a topological space and $S$ a subset of $T$.
The set of all interior points of $S$ is denoted by $S^{\circ}$.

\begin{Claim}
\label{claim:thm:vol:base:cond:6}
Let $C$ be a closed convex set in $\RR^d$. 
For $a \in \RR$, we set $C(a) = \{ x \in C \mid p(x) < a \}$,
where $p : \RR^d \to \RR$ is the projection to the first factor, that is, 
$p(x_1, \ldots, x_d) = x_1$.
If $C^{\circ} \not= \emptyset$ and $C(a) \not= \emptyset$ for some $a \in \RR$,
then $C(a)^{\circ} \not= \emptyset$. 
\end{Claim}

\begin{proof}
Let us choose $x \in C(a)$.
We assume that  $C(a)^{\circ} = \emptyset$. 
Then, as $C(a)$ is a convex set, by \cite[Corollary~2.3.2]{Web},
there is a hyperplane $H$ such that $C(a) \subseteq H$.
Moreover, as $C^{\circ} \not= \emptyset$, there is $y \in C \setminus H$.
Note that $p(y) \geq a$, so that 
\[
0 < (a-p(x))/(p(y) - p(x)) \leq 1.
\]
Here we choose $t \in \RR$ with $0 < t < (a-p(x))/(p(y) - p(x))$.
Then $(1-t)x + ty \in C \setminus H$ and $p((1-t)x + ty) < a$.
This is a contradiction.
\end{proof}

As $\avol(\overline{D}) > 0$, $\Theta(\overline{D})^{\circ} \not=\emptyset$ by Theorem~\ref{thm:integral:formula}.
Therefore, by Claim~\ref{claim:thm:vol:base:cond:5} and Claim~\ref{claim:thm:vol:base:cond:6}, 
$\Theta(\overline{D})^{\circ} \cap \{ x_1 < \mu \} \not= \emptyset$, and hence
$(\Theta(\overline{D}) \setminus \Theta(\overline{D};V_{\bullet}))^{\circ} \not=\emptyset$ because
\[
(\Theta(\overline{D}) \cap \{ x_1 < \mu \})^{\circ} \subseteq
(\Theta(\overline{D}) \setminus \Theta(\overline{D};V_{\bullet}))^{\circ}.
\]
Moreover, note that 
\[
\Theta(\overline{D})^{\circ} = \left\{ x \in \Delta(D) \mid G_{\overline{D}}(x) > 0 \right\}^{\circ}
\subseteq \left\{ x \in \Delta(D) \mid G_{\overline{D}}(x) > 0 \right\}
\]
(cf. \cite[Corollary~2.3.9]{Web}).
Further, 
\[
\vol(\widehat{\Delta}(\overline{D})) = \int_{\Theta(\overline{D})} G_{\overline{D}}(x) dx
\quad\text{and}\quad
\vol(\widehat{\Delta}(\overline{D};V_{\bullet})) = \int_{\Theta(\overline{D};V_{\bullet})} G_{(\overline{D};V_{\bullet})}(x) dx.
\]
Therefore,
\[
\vol(\widehat{\Delta}(\overline{D})) = \int_{\Theta(\overline{D}) \setminus \Theta(\overline{D};V_{\bullet})} G_{\overline{D}}(x) dx + \int_{\Theta(\overline{D};V_{\bullet})} G_{\overline{D}}(x) dx > \vol(\widehat{\Delta}(\overline{D};V_{\bullet})).
\]
Thus \eqref{eqn:thm:vol:base:cond:1} follows from Theorem~\ref{thm:integral:formula}.

\renewcommand{\theTheorem}{\arabic{section}.\arabic{subsection}.\arabic{Theorem}}
\renewcommand{\theClaim}{\arabic{section}.\arabic{subsection}.\arabic{Theorem}.\arabic{Claim}}
\renewcommand{\theequation}{\arabic{section}.\arabic{subsection}.\arabic{Theorem}.\arabic{Claim}}

\section{Applications of Theorem~\ref{thm:vol:base:cond}}

In this section, let us study several applications of Theorem~\ref{thm:vol:base:cond}.

\subsection{Impossibility of Zariski decomposition}
\label{subsec:imp:Zariski}
\setcounter{Theorem}{0}
Let $X$ be a $(d+1)$-dimensional, generically smooth, normal and projective arithmetic variety
(cf. Conventions and terminology~\ref{CT:arith:var}).
A {\em Zariski decomposition of $\overline{D}$} is a decomposition $\overline{D} = \overline{P} + \overline{N}$ 
such that 

\begin{enumerate}
\renewcommand{\labelenumi}{(\arabic{enumi})}
\item
$\overline{P}$ is a nef arithmetic $\RR$-Cartier divisor of $C^0$-type,

\item
$\overline{N}$ is an effective arithmetic $\RR$-Cartier divisor of $C^0$-type, and that

\item
$\avol(\overline{P}) = \avol(\overline{D})$. 
\end{enumerate}
The arithmetic $\RR$-Cartier divisor
$\overline{P}$ (resp. $\overline{N}$) is called a {\em positive part} (resp. a {\em negative part}) of $\overline{D}$.
As a nef arithmetic $\RR$-Cartier divisor of $C^0$-type is
pseudo-effective (cf. \cite[Proposition~6.2.1, Proposition~6.2.2 and Proposition~6.3.2]{MoArZariski}),
if $\overline{D}$ has a Zariski decomposition, then
$\overline{D}$ is pseudo-effective.
Let $\Upsilon(\overline{D})$ be the set of all
nef arithmetic $\RR$-Cartier 
divisors $\overline{M}$ of $C^0$-type with $\overline{M} \leq \overline{D}$.
Note that $\widehat{\Gamma}^{\times}_{\RR}(X, \overline{D}) \not= \emptyset$ implies
$\Upsilon(\overline{D}) \not= \emptyset$.
In the paper \cite{MoArZariski}, we proved that if $d=1$, $X$ is regular and 
$\Upsilon(\overline{D}) \not= \emptyset$,
%there is a nef arithmetic $\RR$-Cartier divisor $\overline{M}$ of $C^0$-type on $X$ with
%$\overline{M} \leq \overline{D}$,
then a Zariski decomposition of $\overline{D}$ exists.
Moreover, by \cite[Theorem~3.5.3]{MoD},
if $\overline{D}$ is pseudo-effective and $D$ is numerically trivial on $X_{\QQ}$,
then a Zariski decomposition of $\overline{D}$ exists in the above sense.

\begin{Theorem}
\label{thm:Zariski:decomp:base:point}
Let $\overline{D}$ be a big arithmetic $\RR$-Cartier divisor of $C^0$-type on $X$. If
there is a Zariski decomposition $\overline{D} = \overline{P} + \overline{N}$ of $\overline{D}$,
then
$\mu_{\RR,\xi}(\overline{D}) = \mult_{\xi}(N)$ for all $\xi \in X_{\QQ}$. 
\end{Theorem}

\begin{proof}
Since $\overline{D} \geq \overline{P}$ and $\mu_{\RR,\xi}(\overline{P}) = 0$ (cf. Proposition~\ref{prop:mu:basic}), we have
\[
\mu_{\RR,\xi}(D) \leq \mu_{\RR,\xi}(\overline{P}) + \mult_{\xi}(N) = \mult_{\xi}(N).
\]
Here we assume that $\mu_{\RR,\xi}(D) < \mult_{\xi}(N)$.
If we set $\mu = \mult_{\xi}(N)$,
then
\[
\avol(\overline{D} ; \mu \xi) < \avol(\overline{D})
\]
by Theorem~\ref{thm:vol:base:cond}.
On the other hand, if $\phi \in \widehat{\Gamma}^{\times}(X, n\overline{P})$, then
\[
\mult_{\xi}(nD + (\phi)) = \mult_{\xi}(nP + (\phi) + nN) \geq \mult_{\xi}(nN) = n\mu.
\]
Therefore,
\[
\aH(X, n\overline{P}) \subseteq \aH(X, n\overline{D} ; n\mu \xi).
\]
Thus $\avol(\overline{P}) \leq \avol(\overline{D}; \mu \xi) < \avol(\overline{D})$.
This is a contradiction.
\end{proof}

\begin{Remark}
\label{rem:def:zariski:decomp}
In the papers \cite{MoBig} and \cite{MoArZariski},
we gave the different  kinds of definitions of Zariski decompositions.
Let us recall their definitions.
Let $\overline{D}$ be an arithmetic $\RR$-Cartier divisor of $C^0$-type.
\begin{enumerate}
\renewcommand{\labelenumi}{(\alph{enumi})}
\item
A decomposition $\overline{D} = \overline{P} + \overline{N}$ is
called a {\em Zariski decomposition in the sense of \cite{MoBig}}
if $\Gamma_{\RR}^{\times}(X, \overline{D}) \not= \emptyset$,
 $\overline{P}$ is a nef arithmetic $\RR$-Cartier divisor of $C^0$-type,
 $\overline{N}$ is an effective arithmetic $\RR$-Cartier divisor of $C^0$-type, and
$\mu_{\RR,\Gamma}(\overline{D}) = \mult_{\Gamma}(N)$ for any horizontal prime divisor
$\Gamma$ on $X$. 

\item
In the case where $d = 1$ and $X$ is regular,
we say a decomposition $\overline{D} = \overline{P} + \overline{N}$ is
a {\em Zariski decomposition in the sense of \cite{MoArZariski}}
if $\Upsilon(\overline{D}) \not= \emptyset$ and
$\overline{P}$ yields the greatest element of $\Upsilon(\overline{D})$.
In this case, $\avol(\overline{D}) = \avol(\overline{P})$ by
\cite[Theorem~9.3.4]{MoArZariski}, so that
it is a Zariski decomposition in the sense of this paper.
\end{enumerate}
The interrelations of these definitions can be described as follows:
%Then we have the following relation:
\begin{enumerate}
\renewcommand{\labelenumi}{(\roman{enumi})}
\item
If $\overline{D}$ is big, then
a Zariski decomposition in the sense of this paper is
a Zariski decomposition in the sense of \cite{MoBig}
(cf. Theorem~\ref{thm:Zariski:decomp:base:point}).
%by the above theorem.

\item
We assume that $d =1$ and $X$ is regular.
A Zariski decomposition in the sense of this paper gives rise to
a Zariski decomposition in the sense of \cite{MoArZariski}
without the bigness of $\overline{D}$,
that is,
a Zariski decomposition in the sense of this paper implies $\Upsilon(\overline{D}) \not= \emptyset$,
so that, by \cite[Theorem~9.2.1]{MoArZariski},
 we can find the greatest element of $\Upsilon(\overline{D})$,
which turns out to be the positive part of the Zariski decomposition
in the sense  \cite{MoArZariski}.
Moreover, if $\overline{D}$ is big,
then a Zariski decomposition in the sense of this paper coincides with
a Zariski decomposition in the sense of \cite{MoArZariski} 
(cf. Theorem~\ref{thm:Zariski:decomp:charaterization}).
\end{enumerate}
\end{Remark} 

%In the paper \cite{MoBig}, we considered a decomposition
%$\overline{D} = \overline{P} + \overline{N}$ such that
%$\mu_{\RR,\Gamma}(\overline{D}) = \mult_{\Gamma}(N)$ for any horizontal prime divisor
%$\Gamma$ on $X$.
%The above theorem means that
%a Zariski decomposition in the sense of this  paper implies the decomposition
%treated in the paper \cite[Section~5]{MoBig}.
By the above remark,
%theorem together with \cite[Theorem~5.6]{MoBig},
we have the following corollary.

\begin{Corollary}
\label{cor:impossible:Zariski:decomp}
We suppose that $d \geq 2$ and $X = \PP^d_{\ZZ} (= \Proj(\ZZ[T_0, T_1, \ldots, T_d]))$.
We set 
$H_i := \{ T_i = 0 \}$ and $z_i := T_i/T_0$ for $i=0, \ldots, d$.
For a sequence $\pmb{a} = (a_0, a_1, \ldots, a_d)$ of positive numbers,
we define an $H_0$-Green function $g_{\pmb{a}}$ of ($C^{\infty} \cap \Tpsh$)-type 
on $\PP^d(\CC)$ 
to be
\[
g_{\pmb{a}} := \log (a_0 + a_1 \vert z_1 \vert^2 + \cdots + a_d \vert z_d \vert^2).
\]
We assume that $\overline{D}$ is given by $(H_0, g_{\pmb{a}})$.
If $\overline{D}$ is big and not nef 
\rom{(}i.e., $a_0 + \cdots + a_d > 1$ and
$a_i < 1$ for some $i$\rom{)},
then, for any birational morphism $f : Y \to \PP^d_{\ZZ}$ of generically smooth, normal 
and projective arithmetic varieties,
there is no Zariski decomposition of $f^*(\overline{D})$ on $Y$.
\end{Corollary}

\begin{Remark}
For a non-big pseudo-effective arithmetic $\RR$-Cartier divisor,
to find a Zariski decomposition in the sense of this paper is a non-trivial
problem. This is closely related to the fundamental question raised in
the paper \cite{MoD}.
Here let us consider an example.
We use the same notation as in \cite{MoD}.
We assume that $\overline{D}_{\pmb{a}}$ is pseudo-effective and not big.
Then, by \cite[Corollary~3.6.4, Proposition~3.6.7, Example~3.6.8]{MoD}, we can find $\phi \in \Rat(\PP^n_{\ZZ})^{\times}_{\RR}$ such that
$\overline{D}_{\pmb{a}} + \widehat{(\phi)}_{\RR} \geq 0$. Thus, if we set
$\overline{P} = \widehat{(\phi^{-1})}_{\RR}$ and $\overline{N} = \overline{D}_{\pmb{a}} + \widehat{(\phi)}_{\RR}$,
then the decomposition $\overline{D}_{\pmb{a}} = \overline{P} + \overline{N}$ 
yields a Zariski decomposition.
%then there is a Zariski decomposition on $\PP^n_{\ZZ}$.
%Indeed, by \cite[Corollary~3.6.4, Proposition~3.6.7, Example~3.6.8]{MoD}, we can find $\phi \in \Rat(\PP^n_{\ZZ})^{\times}_{\RR}$ such that
%$\overline{D}_{\pmb{a}} + \widehat{(\phi)}_{\RR} \geq 0$. Thus, if we set
%$\overline{P} = \widehat{(\phi^{-1})}_{\RR}$ and $\overline{N} = \overline{D}_{\pmb{a}} + \widehat{(\phi)}_{\RR}$,
%then the assertion follows.
Note that $\overline{P}$ is not necessarily an arithmetic $\QQ$-Cartier divisor of $C^0$-type.
%element of $\aDiv_{C^0}(X)_{\QQ}$.
For example, in the case where $d=1$, $a_0 + a_1 = 1$ and $a_1 \not\in \QQ$,
$\phi$ is given by $z_1^{a_1}$ and
$\overline{P} = -a_1 \widehat{(z_1)}$.
Moreover, $-a_1 \widehat{(z_1)}$ is the greatest element of $\Upsilon(\overline{D}_{\pmb{a}})$
(cf. \cite[Section~4]{MoBig}).
\end{Remark}

\subsection{Characterization of Zariski decompositions on arithmetic surfaces}
\setcounter{Theorem}{0}
Let $X$ be
a regular projective arithmetic surface and
let $\pi : X \to \Spec(O_K)$ be the Stein factorization of $X \to \Spec(\ZZ)$.
In this subsection,
we study the following characterizations of the Zariski decomposition of
big arithmetic $\RR$-Cartier divisors on $X$.

\begin{Theorem}
\label{thm:Zariski:decomp:charaterization}
Let $\overline{D}$ be a big arithmetic $\RR$-Cartier divisor of $C^0$-type on $X$ and
let $\overline{D} = \overline{P} + \overline{N}$ be a Zariski decomposition of $\overline{D}$,
where $\overline{P}$ is a positive part of $\overline{D}$.
Then $\overline{P}$ gives the greatest element of
\[
\Upsilon(\overline{D}) = \left\{ \overline{M} \mid \text{$\overline{M}$ is a nef arithmetic 
$\RR$-Cartier divisor of $C^0$-type on $X$ with
$\overline{M} \leq \overline{D}$} \right\}.
\]
\end{Theorem}

\begin{proof}
Let us begin with the following claim:

\begin{Claim}
\label{claim:thm:Zariski:decomp:charaterization:1}
Let $\overline{P}$ and $\overline{Q}$ be nef arithmetic $\RR$-Cartier divisors of $C^0$-type.
We assume the following:
\begin{enumerate}
\renewcommand{\labelenumi}{(\arabic{enumi})}
\item
There are an effective vertical $\RR$-Cartier divisor $E$ and
an $F_{\infty}$-invariant non-negative continuous function $u$ on $X(\CC)$ such that
$\overline{Q} = \overline{P} + (E, u)$.

%\item $\deg_{\QQ}(P_{\QQ}) > 0$, that is, the degree of $P$ on $X_{\QQ}$ is positive.

\item $\avol(\overline{P}) = \avol(\overline{Q}) > 0$.
\end{enumerate}
Then $\overline{P} = \overline{Q}$.
\end{Claim}

\begin{proof}
First of all, note that the degree of $P$ on $X_{\QQ}$ is positive and
\[
u \in \langle \Tqpsh(X(\CC)) \cap C^0(X(\CC)) \rangle_{\RR}
\]
(for details, see \cite[SubSection~1.2]{MoD}).
Moreover, by \cite[Proposition~6.4.2]{MoArZariski}, $\avol(\overline{P}) = \adeg(\overline{P}^2)$ and
$\avol(\overline{Q}) = \adeg({\overline{Q}}^2)$, and hence
$\adeg(\overline{P}^2) = \adeg({\overline{Q}}^2)$.
As
\[
\begin{cases}
\adeg({\overline{Q}}^2) = \adeg({\overline{P}}^2) + \adeg(\overline{Q} \cdot (E, u)) +
\adeg(\overline{P} \cdot (E, u)), \\
\adeg(\overline{Q} \cdot (E, u)) \geq 0, \\
\adeg(\overline{P} \cdot (E, u)) \geq 0,
\end{cases}
\]
we have
\[
\adeg(\overline{Q} \cdot (E, u)) = \adeg(\overline{P} \cdot (E, u)) = 0,
\]
which yields
\[
\adeg(\overline{P} \cdot (E, u)) = \adeg((E, u)^2) = 0.
\]
On the other hand, by virtue of \cite[Proposition~2.1.1]{MoD},
\[
\begin{cases}
{\displaystyle \adeg(\overline{P} \cdot (E, u)) = \adeg(\overline{P} \cdot (E, 0)) + \frac{1}{2} \int_{X(\CC)} c_1(\overline{P}) u,} \\
{\displaystyle \adeg((E, u)^2) = \adeg((E, 0)^2) + \frac{1}{2} \int_{X(\CC)} u dd^c([u]).}
\end{cases}
\]
Therefore, by using Zariski's lemma and \cite[Proposition~1.2.4, (3)]{MoD},
\[
\int_{X(\CC)} c_1(\overline{P}) u = \int_{X(\CC)} u dd^c([u]) = 0
\quad\text{and}\quad\adeg(\overline{P} \cdot (E, 0)) = \adeg((E, 0)^2) = 0.
\]
By the equality condition of \cite[Proposition~1.2.4, (3)]{MoD}, 
$u$ is locally constant, and hence
\[
0 = \int_{X(\CC)} u c_1(\overline{P}) = (\rest{u}{X_1} + \cdots + \rest{u}{X_{[K:\QQ]}})\frac{\deg_{\QQ}(P_{\QQ})}{[K : \QQ]},
\]
where $X_1, \ldots, X_{[K:\QQ]}$ are connected components of $X(\CC)$.
Thus $u = 0$ on $X(\CC)$.
Moreover, by the equality condition of Zariski's lemma,
there are $\frak{p}_1, \ldots, \frak{p}_k \in \Spec(O_K) \setminus \{ 0 \}$ and
$a_1, \ldots, a_k \in \RR_{\geq 0}$ such that
$E = a_1 \pi^{-1}(\frak{p}_1) + \cdots + a_k \pi^{-1}(\frak{p}_k)$.
Thus
\[
0 = \adeg(\overline{P} \cdot (E, 0)) = (a_1 \log\#(O_K/\frak{p}_1) + \cdots + a_k \log(O_K/\frak{p}_k) )\frac{\deg_{\QQ}(P_{\QQ})}{[K:\QQ]},
\]
and hence $a_1 = \cdots = a_k = 0$, that is, $E = 0$, as desired.
\end{proof}

Let us go back to the proof of Theorem~\ref{thm:Zariski:decomp:charaterization}.
As $\overline{D}$ is big, $\Upsilon(\overline{D}) \not= \emptyset$, and hence
we can find the greatest element $\overline{P}_{Zar}$ of $\Upsilon(\overline{D})$
by \cite[Theorem~9.2.1]{MoArZariski}.
Then $\overline{P} \leq \overline{P}_{Zar}$ and $\avol(\overline{P}) = \avol(\overline{P}_{Zar}) = \avol(\overline{D})$. Thus, if we set %$\overline{N} = \overline{D} - \overline{P}$ and
$\overline{N}_{Zar} = \overline{D} - \overline{P}_{Zar}$,
then, by Theorem~\ref{thm:Zariski:decomp:base:point},
\[
\mu_{\RR,\xi}(\overline{D}) = \mult_{\xi}(N) = \mult_\xi(N_{Zar})
\]
for all $\xi \in X_{\QQ}$.  
Therefore,
there are an effective vertical $\RR$-Cartier divisor $E$ and
an $F_{\infty}$-invariant non-negative continuous function $u$ on $X(\CC)$ such that
$\overline{P}_{Zar} = \overline{P} + (E, u)$.
Note that $\overline{P}$ is big. %, and hence $P_{\QQ}$ is big on $X_{\QQ}$. 
Thus,
$\overline{P} = \overline{P}_{Zar}$ by Claim~\ref{claim:thm:Zariski:decomp:charaterization:1}.
\end{proof}

As a corollary of Theorem~\ref{thm:Zariski:decomp:charaterization},
we have the following stronger version of Claim~\ref{claim:thm:Zariski:decomp:charaterization:1}.

\begin{Corollary}
\label{cor:nef:comp}
Let $\overline{P}$ and $\overline{Q}$ be nef arithmetic $\RR$-Cartier divisors of $C^0$-type.
If $\overline{P} \leq \overline{Q}$ and $0 < \avol(\overline{P}) = \avol(\overline{Q})$,
then $\overline{P} = \overline{Q}$.
\end{Corollary}

\begin{proof}
If we set $\overline{N} = \overline{Q} - \overline{P}$, then $\overline{Q} = \overline{P} + \overline{N}$ is a Zariski decomposition of $\overline{Q}$.
Therefore, by Theorem~\ref{thm:Zariski:decomp:charaterization}, $\overline{P} = \overline{Q}$.
\end{proof}

Theorem~\ref{thm:vol:base:cond} still holds for the regular projective arithmetic surface $X$ without the assumption 
$\xi_1, \ldots, \xi_l \in X_{\QQ}$. Namely we have the following theorem:

\begin{Theorem}
\label{thm:vol:base:cond:surface}
Let $\overline{D}$ be a big arithmetic $\RR$-Cartier divisor of $C^0$-type on $X$ and let
$\xi_1, \ldots, \xi_l \in X$. If
$\mu_i > \mu_{\RR,\xi_i}(\overline{D})$ for some $i$,
then 
$\avol(\overline{D} ; \mu_1 \xi_1, \ldots, \mu_l \xi_l) < \avol(\overline{D})$.
\end{Theorem}

\begin{proof}
As in the proof of Theorem~\ref{thm:vol:base:cond}, it is sufficient to see the following:
\addtocounter{Claim}{1}
\begin{equation}
\label{eqn:thm:vol:base:cond:surface:1}
\text{If $\overline{D}$ is big and $\mu > \mu_{\RR,\xi}(\overline{D})$ for $\xi \in X$,
then $\avol(\overline{D} ; \mu \xi)  < \avol(\overline{D})$.}
\end{equation}
By Theorem~\ref{thm:vol:base:cond}, we may assume that the characteristic of the residue field at $\xi$ is positive.
Let $B$ be the Zariski closure of $\{ \xi \}$ in $X$.
In the same way as Claim~\ref{claim:thm:vol:base:cond:2}, we may also assume that
$B$ is a prime divisor.
Note that $\avol(\overline{D} ; \mu \xi) = \avol(\overline{D} -\mu(B,0))$.
If $\overline{D} -\mu(B,0)$ is not big, then the assertion is obvious, so that we may further assume that
$\overline{D} -\mu(B,0)$ is big. We suppose $\avol(\overline{D} -\mu(B,0)) = \avol(\overline{D})$.
Let $\overline{D} = \overline{P} + \overline{N}$ and $\overline{D} -\mu(B,0) = \overline{P}' + \overline{N}'$ be
the Zariski decompositions of $\overline{D}$ and $\overline{D} - \mu(B,0)$ respectively.
As $\overline{P}' \leq \overline{D} - \mu(B, 0) \leq \overline{D}$, we have $\overline{P'} \leq \overline{P}$.
Moreover, 
\[
\avol(\overline{P}') = \avol(\overline{D} - \mu(B,0)) = \avol(\overline{D}) = \avol(\overline{P}).
\]
Thus, by Corollary~\ref{cor:nef:comp}, we obtain $\overline{P}' = \overline{P}$, which implies
\[
\overline{N}' + \mu(B, 0) = \overline{N}.
\]
In particular, $\mult_B(N) \geq \mu$.
On the other hand, by \cite[Claim~9.3.5.1]{MoArZariski},
$\mu_{\RR, \xi}(\overline{D}) = \mult_{B}(N)$, and hence $\mu_{\RR, \xi}(\overline{D}) \geq \mu$.
This is a contradiction.
\end{proof}

\end{document}